%% file: Sparse-Matrix-Testing-122717.tex
\documentclass[letterpaper,11pt]{article}

\usepackage[margin=1in]{geometry}  

\usepackage{enumerate,makecell}
\usepackage[fixpdftex]{xcolor}
\usepackage[title]{appendix}

\usepackage{amsmath} 
\usepackage{amssymb}
\usepackage{amsfonts}
\usepackage{amsthm}
\usepackage{mathrsfs}
\usepackage{color, xcolor}
\usepackage{prettyref}
\usepackage{indentfirst}
\usepackage{graphicx}
\usepackage{xspace}
\usepackage{subfigure}

\usepackage[
            CJKbookmarks=true,
            bookmarksnumbered=true,
            bookmarksopen=true,
            colorlinks=true,
            citecolor=red,
            linkcolor=blue,
            anchorcolor=red,
            urlcolor=blue,
            pdfauthor={T. Tony Cai and Yihong Wu}
            ]{hyperref}

\usepackage{pgf}
\usepackage{authblk}
\usepackage[percent]{overpic}
\usepackage{rotating}

\usepackage{tikz}
\usetikzlibrary{arrows}
\tikzstyle{int}=[draw, fill=blue!20, minimum size=2em]
\tikzstyle{dot}=[circle, draw, fill=blue!20, minimum size=2em]
\tikzstyle{init} = [pin edge={to-,thin,black}]

\newcommand{\Hyper}{{\rm Hypergeometric}}

\newcommand{\stleq}{\stackrel{\rm s.t.}{\leq}}

\newcommand{\PSD}{\sfS}

\newcommand{\eg}{e.g.\xspace}
\newcommand{\ie}{i.e.\xspace}
\newcommand{\iid}{i.i.d.\xspace}

\newcommand{\ones}{\mathbf 1}

\newcommand{\reals}{{\mathbb{R}}}
\newcommand{\integers}{{\mathbb{Z}}}
\newcommand{\naturals}{{\mathbb{N}}}

\newcommand{\supp}{{\rm supp}}
\newcommand{\eexp}{{\rm e}}

\newcommand{\diff}{{\rm d}}

\newcommand{\rank}{\mathop{\sf rank}}
\newcommand{\sr}{\mathop{\sf sr}}
\newcommand{\Tr}{\mathop{\sf Tr}}

\newcommand{\Expect}{\mathbb{E}}
\newcommand{\expect}[1]{\mathbb{E}\left[ #1 \right]}
\newcommand{\eexpect}[1]{\mathbb{E}[ #1 ]}
\newcommand{\expects}[2]{\mathbb{E}_{#2}\left[ #1 \right]}

\newcommand{\Prob}{\mathbb{P}}
\newcommand{\prob}[1]{{ \mathbb{P}\left\{ #1 \right\} }}

\newcommand{\eqdistr}{{\stackrel{\rm (d)}{=}}}
\newcommand{\iiddistr}{{\stackrel{\text{\iid}}{\sim}}}

\newcommand{\Bern}{\text{Bern}}
\newcommand{\Binomial}{\text{Binom}}
\newcommand{\Binom}{\text{Binom}}

\newcommand{\Th}{{^{\rm th}}}
\newcommand{\diverge}{\to \infty}

\theoremstyle{plain}

\newtheorem{lemma}{Lemma}
\newtheorem{theorem}{Theorem}

\theoremstyle{definition}
\newtheorem{definition}{Definition}

\newtheorem{remark}{Remark}

\theoremstyle{plain}

\newrefformat{eq}{(\ref{#1})}
\newrefformat{chap}{Chapter~\ref{#1}}
\newrefformat{sec}{Section~\ref{#1}}
\newrefformat{algo}{Algorithm~\ref{#1}}
\newrefformat{fig}{Figure~\ref{#1}}
\newrefformat{tab}{Table~\ref{#1}}
\newrefformat{rmk}{Remark~\ref{#1}}
\newrefformat{clm}{Claim~\ref{#1}}
\newrefformat{def}{Definition~\ref{#1}}
\newrefformat{cor}{Corollary~\ref{#1}}
\newrefformat{lmm}{Lemma~\ref{#1}}
\newrefformat{prop}{Proposition~\ref{#1}}
\newrefformat{app}{Appendix~\ref{#1}}
\newrefformat{ex}{Example~\ref{#1}}
\newrefformat{exer}{Exercise~\ref{#1}}
\newrefformat{soln}{Solution~\ref{#1}}

\newcommand{\ceil}[1]{{\left\lceil {#1} \right \rceil}}

\newcommand{\norm}[1]{\left\|{#1} \right\|}

\newcommand{\fnorm}[1]{\|#1\|_{\rm F}}
\newcommand{\opnorm}[1]{\left\| #1 \right\|_2}
\newcommand{\Opnorm}[1]{\| #1 \|_2}

\newcommand{\iprod}[2]{\left \langle #1, #2 \right\rangle}
\newcommand{\Iprod}[2]{\langle #1, #2 \rangle}

\newcommand{\indc}[1]{{\mathbf{1}{\left\{{#1}\right\}}}}

\def\innergetnumber#1[#2]#3{#2}
\def\getnumber{\expandafter\innergetnumber\jobname}

\newcommand{\bbS}{{\mathbb{S}}}

\newcommand{\sfS}{{\mathsf{S}}}

\newcommand{\calE}{{\mathcal{E}}}
\newcommand{\calF}{{\mathcal{F}}}

\newcommand{\calM}{{\mathcal{M}}}

\newcommand{\tb}{\tilde{b}}
\newcommand{\tu}{{\tilde{u}}}
\newcommand{\tv}{{\tilde{v}}}

\newcommand{\tB}{{\tilde{B}}}

\newcommand{\tE}{{\tilde{E}}}

\newcommand{\tH}{{\tilde{H}}}
\newcommand{\tI}{{\tilde{I}}}
\newcommand{\tJ}{{\tilde{J}}}

\newcommand{\tM}{{\tilde{M}}}

\newcommand{\tS}{{\tilde{S}}}

\newcommand{\tU}{{\tilde{U}}}

\newcommand{\tX}{{\tilde{X}}}

\newcommand{\comp}[1]{{#1^{\rm c}}}

\newcommand{\ntok}[2]{{#1,\ldots,#2}}

\newcommand{\pth}[1]{\left( #1 \right)}
\newcommand{\qth}[1]{\left[ #1 \right]}
\newcommand{\sth}[1]{\left\{ #1 \right\}}
\newcommand{\bpth}[1]{\Bigg( #1 \Bigg)}
\newcommand{\bqth}[1]{\Bigg[ #1 \Bigg]}

\newcommand{\fracd}[2]{\frac{\diff #1}{\diff #2}}

\newcommand{\TV}{{\rm TV}}

\newcommand{\stepa}[1]{\overset{\rm (a)}{#1}}
\newcommand{\stepb}[1]{\overset{\rm (b)}{#1}}
\newcommand{\stepc}[1]{\overset{\rm (c)}{#1}}

\newcommand{\identity}{\mathbf{I}}
\newcommand{\allones}{\mathbf{J}}

\author{T.~Tony Cai and Yihong Wu\thanks{T.T.~Cai is with the Statistics Department, The Wharton School, University of Pennsylvania, Philadelphia, PA 19104, USA, \url{tcai@wharton.upenn.edu}. The research of T.T.~Cai was supported in part by NSF Grant DMS-1712735 and NIH Grant R01 GM-123056. Y.~Wu is with the Department of Statistics and Data Science, Yale University, New Haven, CT, \url{yihong.wu@yale.edu}. 
The research of Y.~Wu was supported in part by the NSF Grant IIS-1447879, CCF-1527105, and an NSF CAREER award CCF-1651588.
}}

\title{
Statistical and Computational Limits for Sparse Matrix Detection
}

\date{
\today
}

\begin{document}
\DeclareGraphicsExtensions{.pgf}
\graphicspath{{figures/}}

\maketitle

\let\oldthefootnote\thefootnote
\renewcommand{\thefootnote}{\fnsymbol{footnote}}
\let\thefootnote\oldthefootnote

\input{abstract}

\tableofcontents

\input{main}


\end{document}

%% file: abstract.tex
\begin{abstract}

This paper investigates the fundamental limits for detecting a high-dimensional sparse matrix contaminated by white Gaussian noise from both the statistical and computational perspectives. We consider $p\times p$ matrices whose rows and columns are individually $k$-sparse. We provide a tight characterization of the statistical and computational limits for sparse matrix detection, which precisely describe when achieving optimal detection is easy, hard, or impossible, respectively. Although the sparse matrices considered in this paper have no apparent submatrix structure and the corresponding estimation problem has no computational issue at all, the detection problem has a surprising computational barrier when the sparsity level $k$ exceeds the cubic root of the matrix size $p$: attaining the optimal detection boundary is computationally at least as hard as solving the planted clique problem. 

The same statistical and computational limits also hold in the sparse covariance matrix model, where each variable is correlated with at most $k$ others. A key step in the construction of the statistically optimal test is a structural property for sparse matrices, which can be of independent interest.
\end{abstract}

%% file: main.tex
\section{Introduction}
	\label{sec:intro}

The problem of detecting sparse signals arises frequently in a wide range of fields and has been particularly well studied in the Gaussian sequence setting (cf.~the monograph \cite{IS03}). For example, detection of unstructured sparse signals under the Gaussian mixture model was studied in \cite{Ingster97,DJ.HC} for the homoskedastic case and in \cite{CJJ11} for the heteroscedastic case, where sharp  detection boundaries were obtained and adaptive detection procedures proposed.   Optimal detection of structured signals in the Gaussian noise model has also been investigated in \cite{CDH05, CCP11,Cai2014ShortSeg}. One common feature of these vector detection problems is that the optimal statistical performance can always be achieved by computationally efficient procedures such as thresholding or convex optimization.

Driven by contemporary applications, much recent attention has been devoted to inference for high-dimensional matrices, including covariance matrix estimation, principle component analysis (PCA), image denoising, and multi-task learning,  all of which rely on detecting or estimating high-dimensional matrices with low-dimensional structures such as low-rankness or sparsity. 
	For a suite of matrix problems, including sparse PCA \cite{berthet2013lowerSparsePCA}, biclustering \cite{balakrishnan2011tradeoff,MW13b,wang2016statistical,Cai2017Submatrix}, sparse canonical correlation  analysis (CCA) \cite{gao2017sparse} and community detection \cite{HWX14}, 
	a new phenomenon known as \emph{computational barriers} has been recently discovered, which shows that in certain regimes attaining the statistical optimum is computationally intractable, unless the planted clique problem can be solved efficiently.\footnote{The planted clique problem 
	\cite{Alon98} refers to detecting or locating a clique of size $o(\sqrt{n})$ planted in the Erd\"os-R\'enyi graph $G(n,1/2)$.
Conjectured to be computationally intractable \cite{Jer92,Feldman2012statAlg}, this problem has been frequently used as a basis for quantifying hardness of average-case problems \cite{Hazan2011Nash,alon2007testing}.}
	In a nutshell, the source of computational difficulty in the aforementioned problems is their \emph{submatrix sparsity}, where the signal of interests is concentrated on a submatrix within a large noisy matrix. This combinatorial structure provides a direct connection to, and allows these matrix problems  
to be reduced in polynomial time from, the planted clique problem, thereby creating computational gaps for not only the detection but also support recovery and estimation.
	
	In contrast, another sparsity structure for matrices postulates the rows and columns are individually sparse, which has been well studied in  covariance matrix estimation \cite{BJ08b,karoui2008operator,CZ12,fan2015estimation}. The motivation is that in many real-data applications each variable is only correlated with a few others. Consequently, each row and each column of the covariance matrix are individually sparse but, unlike sparse PCA, biclustering, or group-sparse regression, their support sets need not be aligned.
	Therefore this sparsity model does not postulate any submatrix structure of the signal; indeed, it has been shown for covariance matrix estimation that entrywise thresholding of the sample covariance matrix proposed in \cite{BJ08b} attains the minimax estimation rate \cite{CZ12}.
	
	The focus of the present paper is to understand the fundamental limits of detecting sparse matrices from both the statistical and computational perspectives.
	While achieving the optimal estimation rate does not suffer from any computational barrier, it turns out the detection counterpart does when and only when the sparsity level exceeds the \emph{cubic root} of the matrix size. This is perhaps surprising because the sparsity model itself does not enforce explicitly
	any submatrix structure, which has been responsible for problems such as sparse PCA to be reducible from the planted clique. Our main result is a tight characterization of the statistical and computational limits of detecting sparse matrices  in both the Gaussian noise model and the covariance matrix model, which precisely describe when achieving optimal detection is easy, hard, and impossible, respectively.

	
			\subsection{Setup}
	\label{sec:setup}

	We start by formally defining the sparse matrix model:
	\begin{definition}
	\label{def:sparsemat}
		We say a $p\times p$ matrix $M$ is \emph{$k$-sparse} if all of its rows and columns are $k$-sparse vectors, i.e., with no more than $k$ non-zeros. 
	Formally, denote the $i\Th$ row of $M$ by $M_{i*}$ and the $i\Th$ column by $M_{*i}$. The following parameter set
\begin{equation}
\calM(p,k)	= \{M \in \reals^{p \times p}: \|M_{i*}\|_0 \leq k, \|M_{*i}\|_0 \leq k, \forall i \in [p]\}.
	\label{eq:sparsemat}
\end{equation}
denotes the collection of all $k$-spares $p\times p$ matrices, where $\|x\|_0 \triangleq \sum_{i\in[p]} \indc{x_i\neq 0}$ for $x\in\reals^p$.	
	\end{definition}

Consider the following ``signal $+$ noise" model, where we observe a sparse matrix contaminated with Gaussian noise:
\begin{equation}
X = M + Z
\label{eq:GLM}
\end{equation}
where $M$ is a $p \times p$ unknown mean matrix, and $Z$ consists of \iid~entries normally distributed as $N(0,\sigma^2)$.
Without loss of generality, we shall assume that $\sigma=1$ throughout the paper.

Given the noisy observation $X$, the goal is to test whether the mean matrix is zero or a $k$-sparse nonzero matrix, measured in the spectral norm.
Formally, we consider the following hypothesis testing problem:
\begin{equation}
	\begin{cases}
	H_0: & M = 0 \\
H_1: & \|M\|_2 \geq \lambda, \text{ $M$ is $k$-sparse}
	\end{cases}
	\label{eq:htsparse}
\end{equation}
where the mean matrix $M$ belongs to the parameter space
	\begin{equation}
\Theta(p,k,\lambda)	=  	\{M\in \reals^{p \times p}: M \in \calM(p,k), \|M\|_2 \geq \lambda\}.
	\label{eq:sparsemean}
\end{equation}
Here we use the spectral norm $\|\cdot\|_2$, namely, the largest singular value, to measure the signal strength under the alternative hypothesis. It turns out that if we use the Frobenius norm to define the alternative hypothesis, the sparsity structure does not help detection, in the sense that, the minimal $\lambda$ required to detect $1$-sparse matrices is within a constant factor of that in the 
non-sparse case; furthermore, the matrix problem collapses to its vector version (see \prettyref{sec:altf} for details).

For covariance model, the counterpart of the detection problem \prettyref{eq:sparsemean} is the following.
Consider the Gaussian covariance model, where we observe $n$ independent samples drawn from the $p$-variate normal distribution $N(0, \Sigma)$ with an unknown covariance matrix $\Sigma$. 
In the sparse covariance matrix model, each coordinate is correlated with at most $k$ others. Therefore each row of the covariance matrix $\Sigma$ has at most $k$ non-zero off-diagonal entries. This motivates the following detection problem:
\begin{equation}
	\begin{cases}
	H_0: & \Sigma = \identity \\
H_1: & \|\Sigma-\identity\|_2 \geq \lambda, \text{ $\Sigma-\identity$ is $k$-sparse}
	\end{cases}
	\label{eq:htsparse1}
\end{equation}
Under the null hypothesis, the samples are pure noise; under the alternative, there exists at least one significant factor and the entire covariance matrix is $k$-sparse. 
The goal is to determine the smallest $\lambda$ so that the factor can be detected from the samples.

\subsection{Statistical and computational limits}

For ease of exposition, let us focus on the additive Gaussian noise model and consider the following asymptotic regime, wherein the sparsity and the signal level grow polynomially in the dimension as follows:
\[
k = p^{\alpha} \quad \mbox{and} \quad \lambda=p^{\beta}
\]
 with $\alpha \in [0,1]$ and $\beta>0$ held fixed and $p\diverge$. \prettyref{thm:main} in Section \ref{sec:main} implies that the critical exponent of $\lambda$ behaves according to the following piecewise linear function:
\[
\beta^* = \begin{cases}
\alpha & \alpha \leq \frac{1}{3} \\
\frac{1+\alpha}{4} & \alpha \geq \frac{1}{3} 
\end{cases}
\]
in the sense that 
if $\beta > \beta^*$, 
there exists a test that achieves vanishing probability of error of detection uniformly over all $k$-sparse matrices; 
conversely, if $\beta < \beta^*$, 
no test can outperform random guessing asymptotically.

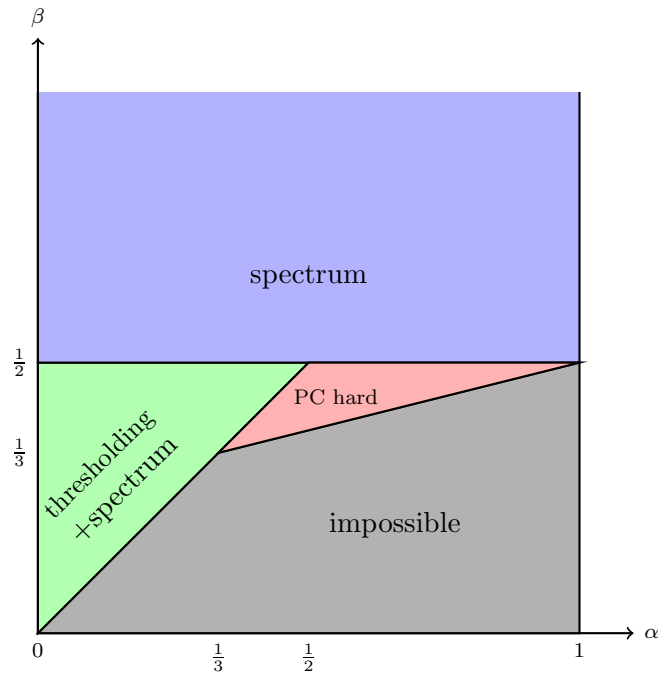
\begin{figure}[ht]
\centering
\begin{tikzpicture}[scale=0.9,domain=0:1,xscale=8,yscale=8, thick]
\draw[->] (0,0) node [below] {\scriptsize 0} -- (0,1/3) node[left] {\scriptsize $\frac{1}{3}$}-- (0,1/2) node[left] {\scriptsize $\frac{1}{2}$}  
-- (0,1.1) node[above] {\scriptsize $\beta$};
\draw[->] (0,0) -- (1/3,0) node[below] {\scriptsize $\frac{1}{3}$}-- (1/2,0) node[below] {\scriptsize $\frac{1}{2}$}-- (1,0) node[below] {\scriptsize 1} -- (1.1,0) node[right] {\scriptsize $\alpha$};
\fill[fill=black!30!, draw=black] (0,0) -- (1/3,1/3) -- (1,1/2) -- (1,0) -- cycle; 
\filldraw[fill=green!30!, draw=black] (0,0) -- (1/2,1/2) -- (0,1/2) -- cycle; 
\fill[fill=blue!30!] (0,1/2) -- (1,1/2) -- (1,1) -- (0,1)-- cycle; 
\draw (0,1)--(0,1/2) -- (1,1/2) -- (1,1); 
\filldraw[fill=red!30, draw=black] (1/3,1/3) -- (1/2,1/2) -- (1,1/2) -- cycle; 
\node at (.8,.2) [left,rotate=0,align=center] {impossible};
\node at (0.03,0.18) [right,align=center,rotate=45] {\small thresholding\\$+$spectrum};
\draw (.5,.7) node [below,align=center] {spectrum};
\draw (.55,.47) node [below,align=center] {\scriptsize PC hard};
\end{tikzpicture}
\caption{Statistical and computational limits in detecting sparse matrices.}%
\label{fig:phase}%
\end{figure}

More precisely, as shown in \prettyref{fig:phase}, the phase diagram of $\alpha$ versus $\beta$ is divided into four regimes:
\begin{enumerate}[(I)]
	\item $\beta > \alpha$: The test based on the largest singular value of the entrywise thresholding estimator succeeds. In particular, we reject if $\Opnorm{X^{\rm Th}} \gtrsim k\sqrt{\log p}$, where $X^{\rm Th}_{ij} = X_{ij} \indc{|X_{ij}| = \Omega(\sqrt{\log p})}$.
	\item $\beta > \frac{1}{2}$: The test based on the large singular value of the direct observation succeeds. In particular, we reject if $\opnorm{X} \gtrsim \sqrt{p}$.
	\item $\frac{1+\alpha}{4} < \beta < \alpha \wedge \frac{1}{2}$: detection is as hard as solving the planted clique problem.
	\item $\beta < \alpha \wedge \frac{1+\alpha}{4}$: detection is information-theoretically impossible.
\end{enumerate}
As mentioned earlier, the computational intractability in detecting sparse matrices is perhaps surprising because 
\begin{itemize}
\item[(a)] achieving the optimal estimation rate does not present any computational difficulty;
\item[(b)] unlike problems such as sparse PCA, the sparse matrix model in \prettyref{def:sparsemat} does not explicitly impose any submatrix sparsity pattern as the rows are individually sparse and need not share a common support.
\end{itemize}

The result in \prettyref{fig:phase} shows that in the moderately sparse regime of $p^{1/3} \ll k \ll p$, outperforming entrywise thresholding is at least as hard as solving planted clique. However, it is possible to improve over entrywise thresholding using computationally inefficient tests. We briefly describe the construction of the optimal test: The first stage is a standard $\chi^2$-test, which rejects the null hypothesis if the mean matrix $M$ has a large Frobenius norm. Under the alternative, if the data can survive this test, meaning $\fnorm{M}$ is small, then $M$ has small stable rank (i.e. $\fnorm{M}^2/\|M\|_2^2$) thanks to the assumption that $\|M\|_2$ is large. The key observation is that for sparse matrices with small stable rank there exists a sparse \emph{approximate singular vector} $v$, in the sense that $\|Mv\|\gtrsim \|M\|_2\|v\|$. Then in the second stage we perform a scan test designed in the similar spirit as in detecting submatrices or sparse principle components. 
The key structural property of sparse matrices is established using a celebrated result of Rudelson and Vershynin \cite{RV07} in randomized numerical linear algebra which shows that the Gram matrix of any matrix $M$ of low stable rank can be approximated by that of a small submatrix of $M$. This shows the existence of sparse approximate singular vector by means of probabilistic method but does not provide a constructive method to find it, which, as  \prettyref{fig:phase} suggests, is likely to be computationally intractable.

To conclude this part, we mention that, the same statistical and computational limits in \prettyref{fig:phase} also apply to detecting sparse covariance matrices when $\lambda$ is  replaced by $\lambda \sqrt{n}$, under appropriate assumptions on the sample size; see \prettyref{sec:computational} for details. 

\subsection{Related work}
	\label{sec:related}

	As opposed to the vector case, there exist various notions of sparsity for matrices as motivated by specific applications, such as:	
	\begin{itemize}
		\item Vector sparsity:
		the total number of nonzeros in the the matrix is constrained \cite{candes2011robust}, e.g., in robust PCA.
		\item Row sparsity: each row of the matrix is sparse, e.g. matrix denoising \cite{klopp2015estimation}.
		\item Group sparsity: each row of the matrix is sparse and shares a common support, e.g., group-sparse regression \cite{lounici2011oracle}.
		\item Submatrix sparsity: the matrix is zero except for a small submatrix, e.g., sparse PCA \cite{BR12,CMW12}, biclustering \cite{BI12,balakrishnan2011tradeoff,MW13b,wang2016statistical}, sparse SVD \cite{yang2016rate}, sparse CCA \cite{gao2017sparse}, and community detection \cite{HajekWuXu14}.
	\end{itemize}
	The sparse matrix model (\prettyref{def:sparsemat}) studied in this paper is stronger than the vector or row sparsity and weaker than submatrix sparsity.
	
	The statistical and computational aspects of detecting matrices with submatrix sparsity has been investigated in the literature for the Gaussian mean, covariance and the Bernoulli models. 
	In particular, for the spiked covariance model where the leading singular vector is assumed to be sparse, 
the optimal detection rate has been obtained in \cite{BR12,CMW13}. 	Detecting submatrices in additive Gaussian noise was studied by Butucea and Ingster \cite{BI12} who not only found the optimal rate but also determined the sharp constants. 
	In the random graph (Bernoulli) setting, the problem of detecting the presence of a small denser community planted in an Erd\"os-R\'enyi graph was studied in \cite{arias2013community}; here the entry of the mean adjacency matrix is $p$ on a small submatrix and $q<p$ everywhere else.
	The computational lower bounds in all three models were established in \cite{berthet2013lowerSparsePCA,MW13b,HajekWuXu14} by means of reduction to the planted clique problem.


 Another work that is closely related to the present paper is  \cite{arias2012detection,ABL12}, where the goal is to detect covariance matrices with sparse correlation. Specifically, in the $n$-sample Gaussian covariance model, the null hypothesis is the identity covariance matrix and the alternative hypothesis consists of covariances matrices whose off-diagonals are equal to a positive constant on a submatrix and zero otherwise. 
Assuming various combinatorial structure of the support set, the optimal tradeoff between the sample size, dimension, sparsity and the correlation level has been studied.
	Other work on testing high-dimensional covariance matrices that do not assume sparse alternatives
include testing independence and sphericity, with specific focus on asymptotic power analysis and the limiting distribution of 
test statistics \cite{chen2010tests,CM13, onatski2013asymptotic,onatski2014signal}.
Finally, we mention that yet another two-dimensional detection problem in Gaussian noise \cite{CCHZ08}, where the sparse alternative corresponds to paths in a large graph.

\subsection{Organization and notations}
	\label{sec:notation}
		
We introduce the main notation used in this paper:
For any sequences $\{a_n\}$ and $\{b_n\}$ of positive numbers, we write $a_n \gtrsim b_n$ if $a_n\geq cb_n$ holds for all $n$ and some absolute constant $c > 0$, $a_n\lesssim b_n$ if $a_n \gtrsim b_n$, and $a_n \asymp b_n$ if both $a_n\gtrsim b_n$ and $a_n\lesssim b_n$ hold.
In addition, we use $\asymp_k$ to indicate that the constant depends only on $k$.

For any $q\in [1,\infty]$, the $\ell_q\to\ell_q$ induced operator norm of an matrix $M$ is defined as $\|M\|_q \triangleq \max_{\|x\|_{\ell_q} \leq 1} \|Mx\|_{\ell_q} $. In particular, $\|M\|_2$ is the spectral norm, i.e., the largest singular value of $M$, and  $\|M\|_1$ (resp.~$\|M\|_\infty$) is the largest $\ell_1$-norm of the columns (resp.~rows) of $M$. For any $p\times p$ matrix $M$ and $I,J \subset [p]$, let $M_{IJ}$ denote the submatrix $(M_{ij})_{i\in I,j\in J}$. Let $\identity$ and $\allones$ denote the identity and the all-one matrix. Let $\ones$ denote the all-one vector. Let $\PSD_p$ denotes the set of $p\times p$ positive-semidefinite matrices.

\medskip
The rest of the paper is organized as follows: \prettyref{sec:main} presents the main results of the paper in terms of the minimax detection rates for both the Gaussian noise model and the covariance matrix model. Minimax upper bounds together with the testing procedures for the mean model are presented in \prettyref{sec:sparse}, shown optimal by the lower bounds in \prettyref{sec:lb}; in particular, \prettyref{sec:approxvec} introduces a structural property of sparse matrices which underpins the optimal tests in the moderately sparse regime. 
 Results for the covariance model are given in \prettyref{sec:cov} together with additional proofs. \prettyref{sec:computational} discusses the computational aspects and explains how to deduce the computational limit in \prettyref{fig:phase} from that of submatrix detection and sparse PCA. \prettyref{sec:discuss} concludes the paper with a discussion on related problems.

%
%
%
%
%

\section{Main results}
	\label{sec:main}
	
	
	We begin with the Gaussian noise model. 
		To quantify the fundamental limit of the hypothesis testing problem \prettyref{eq:htsparse}, we define $\epsilon^*(p,k,\lambda)$ as the optimal sum of Type-I and Type-II probability of error:
\begin{equation}
	\epsilon^*(p,k,\lambda) = 
	\inf_{\phi}  	\sth{\Prob_0(\phi=1) + \sup_{M \in \Theta(p,k,\lambda)}  \Prob_{M}(\phi=0)  }
	\label{eq:epstar}
\end{equation}
where $\Prob_M$ denotes the distribution of the observation $X=M+Z$ conditioned on the mean matrix $M$, and the infimum is taken over all decision rules $\phi: \reals^{p\times p} \to \{0,1\}$.

Our main result is a tight characterization of the optimal detection threshold for $\lambda$. We begin with the Gaussian noise model.
Define the following upper and lower bounds, which differ by at most a factor of $O\Big(\sqrt{\frac{\log p}{\log\log p}}\Big)$:
\begin{equation}
\lambda_1(k,p)\triangleq
\begin{cases}
k \sqrt{\log p}  & k \leq (\frac{p}{\log p})^{\frac{1}{3}} \\
\pth{kp \log \frac{e p}{k}}^{\frac{1}{4}} & k \geq (\frac{p}{\log p})^{\frac{1}{3}} \\
\end{cases}
\label{eq:ub}
\end{equation}	
and
\begin{equation}
\lambda_0(k,p) \triangleq
\begin{cases}
 k \sqrt{ \log \pth{\frac{p \log p}{k^3}}}  & 	k \leq (p \log p)^{\frac{1}{3}} \\
\pth{kp \log \frac{\eexp p}{k}}^{\frac{1}{4}} & k \geq (p \log p)^{\frac{1}{3}} 
	\end{cases}.
\label{eq:lb}
\end{equation}
\begin{theorem}[Gaussian noise model]
\label{thm:main}
There exists absolute constant $k_0,c_0,c_1$, such that the following holds for all $k_0\leq k \leq p$:
\begin{enumerate}
	\item For any $c>c_1$, if
\begin{equation}
\lambda \geq c \lambda_1(k,p),
\label{eq:main-ub}
\end{equation}
then $\epsilon^*(k,p,\lambda) \leq \epsilon_1(c)$, where $\epsilon_1(c)\to 0$ as $c\diverge$.

\item Conversely, for any $c>c_0$, if 
\begin{equation}
\lambda \leq c \lambda_0(p,k),
\label{eq:main-lb}
\end{equation}
then $\epsilon^*(k,p,\lambda) \geq \epsilon_0(c)-o_{p\to\infty}(1)$, where $\epsilon_0(c)\to 1$ as $c\to0$.

\end{enumerate}

\end{theorem}

To parse the result of \prettyref{thm:main}, let us denote by $\lambda^*(p,k)$ the optimal detection threshold, i.e., the minimal value of $\lambda$ so that the optimal probability of error $\epsilon^*(p,k,\lambda)$ is at most a constant, say, $0.1$. Then we have the following characterization:
\begin{itemize}
\item High sparsity: $k \leq p^{1/3-\delta}$: 
	\[
	\lambda^* \asymp_{\delta}  k \sqrt{\log p}
	\]
\item Moderate sparsity: $k \gtrsim (p \log p)^{1/3}$: 
	\[
	\lambda^* \asymp  \pth{kp  \log \frac{ep}{k} }^{\frac{1}{4}}
	\]
	\item Boundary case: $(\frac{p}{\log p})^{1/3} \lesssim k \lesssim (p \log p)^{1/3}$: 
		\[
k \sqrt{\log \frac{ep\log p}{k^3} } 	\lesssim \lambda^* \lesssim  \pth{kp  \log \frac{ep}{k} }^{\frac{1}{4}},
	\]
	where the upper and lower bounds are within a factor of $O\Big(\sqrt{\frac{\log p}{\log \log p}}\Big)$.
\end{itemize}
Furthermore, two generalizations of \prettyref{thm:main} will be evident from the proof:
(a) the upper bound in \prettyref{thm:main} as well as the corresponding optimal tests apply as long as the noise matrix consists of independent entries with subgaussian distribution with constant proxy variance;
(b) the lower bound in \prettyref{thm:main} continues to hold up even if the mean matrix is constrained to be symmetric. Thus, symmetry does not improve the minimax detection rate.

Next we turn to the sparse covariance model:
Given $n$ independent samples drawn from $N(0,\Sigma)$,  the goal is to test the following hypothesis 
	\begin{equation}	
\begin{cases}
H_0: & ~ \Sigma=\identity	\nonumber \\
H_1: & ~ \Sigma \in \Xi(p,k,\lambda,\tau),
\end{cases}
	\label{eq:htsparse-cov}
	\end{equation}
	where the parameter space for sparse covariances matrices is
	\begin{equation}
	\Xi(p,k,\lambda,\tau) = 
	\{\Sigma \in \PSD_p: \Sigma \in \calM(p,k), \norm{\Sigma-\identity}_2 \geq \lambda, \norm{\Sigma} \leq \tau \}.
	\label{eq:sparsecov}
	\end{equation}	
In other words, under  the alternative, the covariance is equal to identity plus a sparse perturbation. Throughout the paper, the parameter $\tau$ is assumed to be a constant.

Define the minimax probability of error as:
\begin{equation}
	\epsilon_n^*(p,k,\lambda) = 
	\inf_{\phi}  	\sth{\Prob_{\identity}(\phi=1) + \sup_{
	\Sigma \in \Xi(p,k,\lambda,\tau)}  \Prob_{\Sigma}(\phi=0)  }
\end{equation}
where $\phi \in \{0,1\}$ is a function of the samples $(X_1,\ldots,X_n) \iiddistr N(0,\Sigma)$.

Analogous to \prettyref{thm:main}, the next result characterizes the optimal detection threshold for sparse covariance matrices.
\begin{theorem}[Covariance model]
	\label{thm:main-cov}
There exists absolute constants $k_0,C,c_0,c_1$, 
such that the following holds for all $k_0\leq k \leq p$.
\begin{enumerate}
	\item 
Assume that $n \geq C \log p$. For any $c>c_1$, if
\begin{equation}
\lambda  \geq \frac{c}{\sqrt{n}} \lambda_1(k,p),
\label{eq:main-ub-cov}
\end{equation}
then $\epsilon_n^*(k,p,\lambda) \leq \epsilon_1(c)$, where $\epsilon_1(c)\to 0$ as $c\diverge$.
 
\item Assume that
\begin{equation}
n \geq C\lambda_0(p,k)^2 \log p
\label{eq:nassumption-basic}
\end{equation}
and
\begin{equation}
n \geq C\cdot
\begin{cases}
\frac{k^6}{p}  \pth{\frac{p}{k^3}}^{2 \delta} \log^2 p &  k\leq p^{1/3}\\
p & k\geq p^{1/3}
\end{cases},
\label{eq:nassumption-extra}
\end{equation}
where $\delta$ is any constant in $(0, \frac{2}{3}]$.
If
\begin{equation}
\lambda \leq \frac{c}{\sqrt{n}} \lambda_0(k,p),
\label{eq:main-lb-cov}
\end{equation}
then $\epsilon^*(k,p,\lambda) \geq \epsilon_0(c)-o_{p\to\infty}(1)$, where $\epsilon_0(c)\to 1$ as $c\to0$

\end{enumerate}

\end{theorem}
In comparison with \prettyref{thm:main}, we note that the rate-optimal lower bound in \prettyref{thm:main-cov} holds under the assumption that the sample size is sufficiently large. In particular, the condition \prettyref{eq:nassumption-basic} is very mild because, by the assumption that $\|\Sigma\|_2$ is at most a constant, in order for the right hand side of \prettyref{eq:main-lb-cov} to be bounded, it is necessary to have $n \geq \lambda_0(p,k)^2$. 
The extra assumption 
\prettyref{eq:nassumption-extra}, when $k\geq  p^{1/4}$, does impose a non-trivial constraint on the sample size. This assumption is due to the current lower bound technique based on the $\chi^2$-divergence. In fact, the lower bound in \cite{CM13} for testing covariance matrix without sparsity uses the same method and also requires $n \gtrsim p$.

The results of Theorems \ref{thm:main} and \ref{thm:main-cov} also demonstrate the phenomenon of the \emph{separation of detection and estimation}, which is well-known in the Gaussian sequence model.
	The minimax estimation of sparse matrices has been systematically studied by Cai and Zhou \cite{CZ12} in the covariance model, where it is shown that entrywise thresholding  achieves the minimax rate in the spectral norm loss of $k \sqrt{\frac{\log p}{n}}$ provided that $n \gtrsim k^2 \log^3 p$ and $\log n \lesssim \log p$; similar rate of $k \sqrt{\log p}$ also holds for the Gaussian noise model. In view of this result, an interesting question is whether a ``plug-in'' approach for testing, namely, using the spectral norm of the minimax estimator as the test statistic, achieves the optimal detection rate.
This method is indeed optimal in the very sparse regime of $k \ll p^{1/3}$, but fails to achieve the optimal detection rate in the moderately sparse regime of $k\gg p^{1/3}$, which, in turn, can be attained by a computationally intensive test procedure.
This observation should be also contrasted with the behavior in the vector case. To detect the presence of a $k$-sparse $p$-dimensional vector in Gaussian noise, entrywise thresholding, which is the optimal estimator for all sparsity levels, achieves the minimax detection rate in $\ell_2$-norm when $k \ll \sqrt{p}$, while the $\chi^2$-test, which disregards sparsity, is optimal when $k\gg \sqrt{p}$.

\section{Test procedures and upper bounds}
	\label{sec:sparse}
	
	In this section we consider the two sparsity regimes separately and design the corresponding rate-optimal testing procedures.
	In the highly sparse regime of $k \lesssim (\frac{p}{\log p})^{\frac{1}{3}}$, 
tests based on	componentwise thresholding turns out to achieve the optimal rate of detection. 
In the moderately sparse regime of $k \gtrsim (\frac{p}{\log p})^{\frac{1}{3}}$, chi-squared test combined with the approximate singular vector property in \prettyref{sec:approxvec} is optimal.	

	\subsection{A structural property of sparse matrices}
	\label{sec:approxvec}

	Before we proceed to the construction of the rate-optimal tests, we first present a structural property of sparse matrices, which may be of independent interest.
	Recall that a matrix $M$ is $k$-sparse in the sense of \prettyref{def:sparsemat} if its rows and columns are sparse but need not to have a common support. 
	If we further know that $M$ has low rank, then the row support sets must be highly aligned, and therefore $M$ has a sparse eigenvector.
	The main result of this section is an extension of this result to approximately low-rank matrices and their approximate eigenvectors.


 \begin{definition} 
	\label{def:approxvec}
	We say $v \in \reals^p$ is an $\epsilon$-approximate singular vector of $\Sigma$ if $\norm{\Sigma v}_2 \geq (1-\epsilon) \opnorm{\Sigma} \norm{v}$.
\end{definition}

We also need the notion of \emph{stable rank} (also known as numerical rank): 
\begin{equation}
\sr(M) \triangleq \frac{\fnorm{M}^2}{\opnorm{M}^2}\,,
	\label{eq:sr}
\end{equation}
which is always a lower bound of $\rank(M)$.


The following lemma gives sufficient conditions for a sparse matrix to have sparse approximate singular vectors.
The key ingredient of the proof is a celebrated result of Rudelson-Vershynin \cite{RV07} in randomized numerical linear algebra which shows that the Gram matrix of any matrix $M$ of stable rank at most $r$ can be approximated by that of a submatrix of $M$ formed by $O(r \log r)$ rows.
The following is a restatement of \cite[Theorem 3.1]{RV07} without the normalization:
\begin{lemma}
\label{lmm:RV07}	
There exists an absolute constant $C_0$ such that the following holds.
	Let $y\in\reals^n$ be a random vector with covariance matrix $K=\Expect[yy^\top]$. Assume that $\|y\|_2 \leq M$ holds almost surely.
Let $y_1,\ldots,y_d$ be iid copies of $y$. Then
\[
\Expect\bigg\|\frac{1}{d}\sum_{i=1}^d y_i y_i^\top- K\bigg\|_2 \leq C_0 M \sqrt{ \opnorm{K} \frac{\log d}{d}},
\]
provided that the right-hand side is less than $\opnorm{K}$.
\end{lemma}

\begin{theorem}[Concentration of operator norm on small submatrices]
Let $k \in [p]$. Let $M$ be a $p \times p$ $k$-sparse matrix (not necessarily symmetric) in the sense that all rows and columns are $k$-sparse.
Let $r=\sr(M)$. 
Then there exists $I,J\subset [p]$, such that
\[
\opnorm{M_{IJ}} \geq \frac{1}{8} \opnorm{M}, \quad |I| \leq C k r, \quad |J| \leq C k r \log r
\]
where $C$ is 
an absolute constant.
	\label{thm:approxvec}
\end{theorem}

\begin{remark}
The intuition behind the above result is the following: consider the ideal case where $X$ is low-rank, say, $\rank(X)\leq r$. Then its right singular vector belongs to the span of at most $r$ rows and is hence $kr$-sparse; so is the left singular vector. \prettyref{thm:approxvec} extends this simple observation to stable rank with an extra log factor.
Furthermore, the result in \prettyref{thm:approxvec} cannot be improved beyond this log factor. To see this, consider a matrix $M$ consisting of an $m\times m$ submatrix with independent $\Bern(q)$ entries and zero elsewhere, where $q = k/(2m) \ll 1$. Then with high probability, $M$ is $k$-sparse, 
$\opnorm{M} \approx q m$, and $\fnorm{M}^2 \approx q m^2$. 
Although the rank of $M$ is approximately $m$, its stable rank is much lower 
$\sr(M) \approx \frac{1}{q}$, and the leading singular vector of $M$ is $m$-sparse, with $m = \Theta(k \sr(M))$.
In fact, this example plays a key role in constructing the least favorable prior for proving the minimax lower bound in \prettyref{sec:lb}.
\end{remark}

\begin{proof}
Denote the $i\Th$ row of $M$ by $M_{i*}$. 
Denote the $j\Th$ row of $M$ by $M_{*j}$. 
Let  
\begin{align*}
I_0  \triangleq  & ~ \sth{i \in [p]: \|M_{i*}\|_2 \geq \tau }	\nonumber \\
J_0 \triangleq & ~ 	\sth{j \in [p]: \|M_{*j}\|_2 \geq \tau },
\end{align*}
where $\tau > 0$ is to be chosen later. Then 
\begin{equation}
 |I_0| \vee   |J_0| \leq \frac{\fnorm{M}^2}{\tau^2}.
	\label{eq:Icard}
\end{equation}
Since the operator norm and Frobenius norm are invariant under permutation of rows and columns, we may and will assume that $I_0,J_0$ corresponds to the first few rows or columns of $M$. Write $M = \pth{\begin{smallmatrix} A & C \\ D & B \end{smallmatrix}}$ where $B=M_{\comp{I_0} \comp{J_0}}$. 
Since each row of $B$ is $k$-sparse, by the Cauchy-Schwartz inequality its $\ell_1$-norm is at most $\sqrt{k}\tau$. Consequently its $\ell_\infty\to\ell_\infty$ operator norm satisfies $\|B\|_\infty = \max_i \|B_{i*}\|_1 \leq \sqrt{k} \tau$. Likewise, $\|B\|_1 = \max_j \|B_{*j}\|_1 \leq \sqrt{k} \tau$. By duality (see, \eg, \cite[Corollary 2.3.2]{matrix.computation}), 
\begin{align}
\opnorm{B}
\leq \sqrt{\|B\|_1 \|B\|_{\infty}} \leq \sqrt{k} \tau.  \label{eq:B3}
\end{align}
Let $X = (A~C)$ and 
$Y = \pth{\begin{smallmatrix} A \\ D \end{smallmatrix}}$.
By triangle inequality, we have $\opnorm{M} \leq \opnorm{X} + \opnorm{Y} + \opnorm{B}$.
Setting $\tau = \frac{\opnorm{M}}{2 \sqrt{k}}$, we have $\opnorm{B} \leq \opnorm{M}/2$ and hence $\opnorm{X}\vee \opnorm{Y} \geq \frac{\opnorm{M}}{4}$. Without loss of generality, assume henceforth $\opnorm{X} \geq \frac{\opnorm{M}}{4}$. Set $I=I_0$.

Note that $X\in\reals^{\ell \times p}$, where $\ell = |I| \leq \frac{\fnorm{M}^2}{\tau^2} = \frac{4 k\fnorm{M}^2}{\opnorm{M}^2} = 4 L$. Furthermore, 
$\sr(X) = \frac{\fnorm{X}^2}{\opnorm{X}^2} \leq \frac{\fnorm{M}^2}{\opnorm{M}^2/16} = 16 r$.
Next we show that $X$ has a submatrix formed by a few columns whose operator norm is large.
We proceed as in the proof of \cite[Theorem 1.1]{RV07}. 
Write
\[
X=\Bigg[\begin{smallmatrix}x_1^\top\\ \vdots\\ x_\ell^\top\end{smallmatrix}\Bigg], \quad
\tX= \frac{1}{\sqrt{d}}\Bigg[\begin{smallmatrix}y_1^\top\\ \vdots\\ y_d^\top\end{smallmatrix}\Bigg].
\]
Define the random vector $y$ by $\prob{y = \frac{\fnorm{X}}{\|x_i\|_2} x_i } = \frac{\|x_i\|_2^2}{\fnorm{X}^2}$
and let $y_1,\ldots,y_d$ which are iid copies of $y$.
Then $X^\top X=\Expect[yy^\top]$ and $\tX^\top\tX = \frac{1}{d} \sum_{i=1}^d y_i y_i^\top$.
Furthermore, $\|y\|_2 \leq \|X\|_F$ almost surely and $\opnorm{\Expect[yy^\top]} = \opnorm{X}^2$.
By \prettyref{lmm:RV07},
\[
\Expect \opnorm{\tilde X^\top\tilde X-X^\top X} \leq 
C_0 \sqrt{\frac{\log d}{d}} \fnorm{X}\opnorm{X} \leq \frac{1}{4} \opnorm{X}^2,
\]
where the last inequality follows by choosing $d = \ceil{C r \log r}$ with $C$ being a sufficiently large universal constant.
Therefore there exists a realization of $\tilde X$ so that the above inequality holds. 
Let $J$ be the column support of $\tilde X$. Since the rows of $\tilde X$ are scaled version of those of $X$ which are $k$-sparse, we have $|J| \leq d k$.
Let $v$ denote a leading right singular vector of $\tilde X$, \ie, $\tilde X^\top\tilde X v = \Opnorm{\tilde X}^2 v$ and $\|v\|_2=1$.
Then $\supp(v) \subset J$.
Note that
\begin{align*}
\|Xv\|_2^2 
= & ~ v^\top X^\top Xv = v^\top\tX^\top\tX v + v^\top(X^\top X-\tX^\top\tX) v \\
\geq & ~ \Opnorm{\tilde X}^2- \Opnorm{X^\top X-\tX^\top\tX} 	\\
\geq & ~ \Opnorm{X}^2- 2 \Opnorm{X^\top X-\tX^\top\tX} 	\\
\geq & ~ \frac{1}{2} \Opnorm{X}^2.
\end{align*}
 Therefore $\opnorm{X_{*J}} \geq \|Xv\|_2 \geq \frac{1}{\sqrt{2}} \opnorm{X} \geq \frac{1}{4\sqrt{2}} \opnorm{M}$.
The proof is completed by noting that $X_{*J} = M_{I J}$.
\end{proof}

\subsection{Highly sparse regime}


It is has been shown that, in the covariance model, entrywise thresholding is rate-optimal for estimating the matrix itself with respect to the spectral norm \cite{CZ12}. 
It turns out that in the very sparse regime entrywise thresholding	is optimal for testing as well.
	Define 
	\[
	\hat M = (X_{ij} \indc{|X_{ij}| \geq \tau}).
	\]
and the following test
\begin{align}
\psi(X) = \indc{\|\hat M\|_2 \geq \lambda}
\label{eq:thtest} 
\end{align}

\begin{theorem}
	\label{thm:ub.ksmall}
For any $\epsilon \in (0,1)$, if
	\begin{equation}
	\lambda > 2 k \sqrt{2 \log \frac{4p^2}{\epsilon}}
	\label{eq:ub.ksmall}
\end{equation}
then the test \prettyref{eq:thtest} with $\tau  = \sqrt{2 \log \frac{4p^2}{\epsilon}}$ satisfies
\[
\Prob_0(\psi=1) + \sup_{M \in \Theta(p,k,\lambda)}  \Prob_{M}(\psi=0)   \leq \epsilon
\]
for all $1 \leq k \leq p$. 
\end{theorem}
	\begin{proof}
Denote the event $E = \{\|Z\|_{\ell_\infty} \leq \tau\}$.
Conditioning on $E$, for any $k$-sparse matrix $M \in \calM(p,k)$, we have $\hat M \in \calM(p,k)$ and
\begin{equation}
\|\hat M - M\|_2 \leq k\tau.
\label{eq:highprobbound}
\end{equation}
To see this, note that for any $i,j$, $\hat M_{ij}=0$ whenever $M_{ij}=0$. Therefore $\|\hat M_{i*}-M_{i*}\|_{\ell_1} \leq k \|Z\|_{\ell_\infty} \leq k\tau$ and, consequently, 
$\|\hat M-M\|_1 = \max_i \|\hat M_{i*}-M_{i*}\|_{\ell_1} \leq k\tau$. Similarly, 
$\|\hat M-M\|_\infty = \max_j \|\hat M_{*j}-M_{*j}\|_{\ell_1} \leq k\tau$. Therefore \prettyref{eq:highprobbound} follows from the fact that 
$\|\cdot\|_2^2 \leq \|\cdot\|_1 \|\cdot\|_\infty$ for matrix induced norms.
Therefore if $\lambda > 2 k\tau$, then
\[
\Prob_0(\psi=1) + \sup_{M \in \Theta(p,k,\lambda)}  \Prob_{M}(\psi=0)   \leq 2 \prob{\|Z\|_{\ell_\infty} > \tau}
\leq 4 p^2 e^{-\tau^2/2}.
\]
This completes the proof.
\end{proof}

\subsection{Moderately sparse regime}

%



Our test in the moderately sparse regime relies on the existence of sparse approximate eigenvectors established in \prettyref{thm:approxvec}. 
More precisely, the test procedure is a combination of the matrix-wise $\chi^2$-test and the scan test based on the largest spectral norm of $m\times m$ submatrices, which is detailed as follows: Let 
\[
m = C \sqrt{\frac{kp}{\log \frac{\eexp p}{k}}}.
\]
where $C$ is the universal constant from \prettyref{thm:approxvec}. Define the following test statistic
\begin{equation}
	T_m(X) = \max\{\opnorm{X_{IJ}}: I,J \subset [p], |I|=|J|=m\}
	\label{eq:TmX}
\end{equation}
and the test
\begin{align}
\psi(X) = \indc{\fnorm{X}^2 \geq p^2 + s} \vee \indc{T_m(X) \geq t}
\label{eq:wtest} 
\end{align}
where 
\begin{equation}
s \triangleq 2 \log \frac{1}{\epsilon} + 2 p \sqrt{\log \frac{1}{\epsilon}}, \quad t \triangleq 2\sqrt{m} + 4 \sqrt{m\log \frac{ep}{m}}.
\label{eq:st}
\end{equation}


\begin{theorem}
	\label{thm:ub.kbig}
There exists a universal constant $C_0$ such that the following holds.
For any $\epsilon \in (0,1/2)$, if
	\begin{equation}
	\lambda \geq C_0 \sth{kp \log \frac{1}{\epsilon} \log \pth{\frac{p}{k} \log \frac{1}{\epsilon}  }}^{\frac{1}{4}},
	\label{eq:ub.kbig}
\end{equation}
then the test \prettyref{eq:wtest} satisfies
\[
\Prob_0(\psi=1) + \sup_{M \in \Theta(p,k,\lambda)}  \Prob_{M}(\psi=0)   \leq \epsilon
\]
holds for all $1 \leq k \leq p$.
\end{theorem}

	\begin{proof}
		First consider the null hypothesis, where $M=0$ and $X=Z$ has iid standard normal entries so that $\fnorm{Z}^2-p^2 = O_P(p)$. 
		By standard concentration equality for $\chi^2$ distribution, we have
		\[
		\prob{|\fnorm{Z}^2-p^2| > s } \leq \epsilon,
		\]
		where 
		\[
		s \triangleq 2 \log \frac{1}{\epsilon} + 2 p \sqrt{\log \frac{1}{\epsilon}}.
		\]
		Consequently the false alarm probability satisfies
		\[
		\Prob_0(\psi=1) \leq \underbrace{\prob{\fnorm{Z}^2-p^2 > C_0 p}}_{\leq \epsilon} + \binom{p}{m}^2 \prob{\|W\|_2 \geq t}.
		\]
		where $t = 2\sqrt{m} + 4 \sqrt{m\log \frac{ep}{m}}$ and $W \triangleq Z_{[m],[m]}$.
		By the Davidson-Szarek inequality \cite[Theorem II.7]{Davidson01}, $\opnorm{W} \overset{\text{s.t.}}{\leq} N(2\sqrt{m},1)$. Then
		$\prob{\|W\|_2 \geq t }  \leq (\frac{em}{p})^{m}$. Hence the false alarm probability vanishes.

		Next consider the alternative hypothesis, where, by assumption, $M$ is row/column $k$-sparse and $\|M\| \geq \lambda$.
		To begin, suppose that $\fnorm{M} \geq 2\sqrt{s}$. Then since 
		$\fnorm{X}^2 - p^2 = \fnorm{M}^2 + 2 \iprod{M}{Z} + \fnorm{Z}^2-p^2$, we have
		\begin{align*}
		\Prob\{\fnorm{M+Z}^2 - p^2 < s\} 
		\le & ~  \prob{\fnorm{M}^2 + 2 \iprod{M}{Z} < 2 s } + \prob{\fnorm{Z}^2-p^2 <  -s}  \\
		\leq & ~ 	\exp(-s^2/8) + \epsilon.
		\end{align*}
		Therefore, as usual, if $\fnorm{M}$ is large, the $\chi^2$-test will succeeds with high probability. Next assume that $\fnorm{M} < 2\sqrt{s}$. Therefore $M$ is approximately low-rank:
		\[
		\sr(M) \leq r \triangleq \frac{4s}{\lambda^2}.
		\]
		By \prettyref{thm:approxvec}, there exists an absolute constant $C$ and $I,J\subset [p]$ of cardinality at most 
		\[
		m = C k r \log r = C k \frac{4s}{\lambda^2} \log \frac{4s}{\lambda^2},
		\]
		such that
$\opnorm{M_{IJ}} \geq \frac{1}{8} \lambda$. Therefore the statistic defined in \prettyref{eq:TmX} satisfies
$
T_m(X) \geq \opnorm{X_{IJ}} \geq \frac{\lambda}{8} - \opnorm{Z_{IJ}}$. Therefore
$T_m(X) \geq \frac{\lambda}{8} - 3 \sqrt{m}$ with probability at least $1-\exp(-\Omega(m))$.
Choose $\lambda$ so that
\[
\frac{\lambda}{8} - 3 \sqrt{m} \geq t,
\]
Since $t + 3 \sqrt{m} = 5\sqrt{m} + 4 \sqrt{m\log \frac{ep}{m}} \leq 9 \sqrt{m\log \frac{ep}{m}}$, it suffices to ensure that
$\lambda \geq c_0 \sqrt{m\log \frac{ep}{m}}$ for some absolute constant $c_0$.
Plugging the expression of $m$, we found a sufficient condition is
$\lambda \geq C_0 (ks \log \frac{\eexp s}{k})^{\frac{1}{4}}$ for some absolute constant $C_0$.
The proof is completed by noting that $s \leq 2 p(\log \frac{1}{\epsilon} + \log \frac{1}{\epsilon})$ and $s \mapsto s \log \frac{es}{k}$ is increasing.
	\end{proof}

\section{Minimax lower bound}
	\label{sec:lb}
In this section we prove the lower bound part of \prettyref{thm:main}.
In \prettyref{sec:lb.gen}, we first present a general strategy of deriving minimax lower bound for functional hypothesis testing problems, which involves priors not necessarily supported on the parameter space. To apply this strategy, in \prettyref{sec:prior} we specify a prior under which the matrix is $k$-sparse with high probability. 
The lower bound is proved by bounding the $\chi^2$-divergence between the null distribution and the mixture of the alternatives.

	\subsection{General strategy}
	\label{sec:lb.gen}
We begin by providing a general strategy of constructing lower bounds for composite hypothesis testing problem, which is in particular useful for testing functional values. Given an experiment $\{P_\theta: \theta \in \Theta\}$ and two parameter subsets $\Theta_0, \Theta_1 \subset \Theta$, consider the following composite hypothesis testing problem 
\begin{equation}
H_0: \theta \in \Theta_0 \quad \text{v.s.} \quad H_1: \theta \in \Theta_1	
	\label{eq:HTc}
\end{equation}	
	Define the minimax sum of Type-I and Type-II error probabilities as
	\[
\calE(\Theta_0, \Theta_1) \triangleq	\inf_{A} \sup\{P_{\theta_0}(A) + P_{\theta_1}(\comp{A}): \theta_i \in \Theta_i, i=0,1\},
	\]
	where the infimum is taken over all measurable sets $A$. By the minimax theorem (c.f., \eg, \cite[p.~476]{Lecam86}), the minimax probability error is given by least favorable Bayesian problem:
\begin{equation}
\calE(\Theta_0, \Theta_1) = 1 - \inf\{\TV(P_{\pi_0}, P_{\pi_1}): \pi_0 \in \calM(\Theta_0), \pi_1 \in \calM(\Theta_1)\}.	
	\label{eq:lecam}
\end{equation}
where $\calM(\Theta_i)$ denotes the set of all probability measures supported on $\Theta_i$, and $P_{\pi}(\cdot) = \int P_{\theta}(\cdot) \pi(\diff \theta)$ denotes the mixture induced by the prior $\pi$. Therefore, any pair of priors give rise to a lower bound on the probability of error; however, sometimes it is difficult to construct priors that are \emph{supported} on the respective parameter subsets.
To overcome this hurdle, 
the following lemma, which is essentially the same as \cite[Theorem 2.15 (i)]{Tsybakov09}, allows priors to be supported on a possibly extended parameter space $\Theta'$. For completeness, 
we state a simpler and self-contained proof using the data processing inequality \cite{Csiszar63} and the triangle inequality for total variation.

	\begin{lemma}[Lower bound for testing]
	\label{lmm:sep}
	Let $\Theta'\supset \Theta$. 
	Let $\pi_0, \pi_1$ be priors supported on $\Theta'$. If  
	\begin{equation}
	\TV(P_{\pi_0}, P_{\pi_1}) \leq 1 - \delta
	\label{eq:tvmixp}
\end{equation}
and 
\begin{equation}
\pi_i(\Theta_i) \geq 1 - \epsilon_i, \quad i=0,1,
	\label{eq:prob.sep}
\end{equation}
then
	\begin{equation}
	\calE(\Theta_0,\Theta_1) \geq \delta - \epsilon_0 - \epsilon_1.
	\label{eq:sep}
\end{equation}
\end{lemma}
\begin{proof}
Define the following priors by conditioning: for $i=0,1$, let $\tilde \pi_i = \left.\pi_i\right|_{\Theta_i}$, \ie,  $\tilde \pi_i(\cdot) = \frac{\pi_i(\cdot \cap \Theta_i)}{\pi_i(\Theta_i}$. Then, by triangle inequality,
	\begin{align*}
\TV(P_{\tilde\pi_0}, P_{\tilde\pi_1})
\leq & ~ \TV(P_{\pi_0}, P_{\pi_1}) + \TV(P_{\tilde\pi_0}, P_{\pi_0}) + \TV(P_{\tilde\pi_1}, P_{\pi_1}) 	\\
\stepa{\leq} & ~ \TV(P_{\pi_0}, P_{\pi_1}) + \TV(\tilde\pi_0, {\pi_0}) + \TV(\tilde\pi_1, \pi_1) 	\\
\stepb{\leq} & ~ 1-\delta + \epsilon_0 + \epsilon_1.
\end{align*}
where 
(a) follows from the data-processing inequality (or convexity) of the total variation distance; 
(b) follows from that $\TV(\tilde\pi_i, {\pi_i})=\pi_i(\Theta_i^c)$.
The lower bound \prettyref{eq:sep} then follows from the characterization \prettyref{eq:lecam}.
\end{proof}

\subsection{Least favorable prior}
	\label{sec:prior}
	To apply \prettyref{lmm:sep} to the detection problem \prettyref{eq:htsparse}, we have $\Theta_0=\{0\}$ and $\Theta_1 = \Theta(p,k,\lambda)$, with $\pi_0=\delta_0$ and $\pi_1$ a prior distribution under which the matrix is sparse with high probability.
	Next we describe the prior that leads to the optimal lower bound in \prettyref{thm:main}.
Let $I$ be chosen uniformly at random from all subsets of $[p]$ of cardinality $m$.
Let $u=(\ntok{u_1}{u_p})$ be independent Rademacher random variables. 
Let $B$ be a $p\times p$ matrix with \iid~$\Bern(\frac{k}{m})$ entries and let $(u,I,B)$ be independent. 
Let $U_I$ denote the diagonal matrix defined by $(U_I)_{ii}= u_i \indc{i\in I}$. 
Let $t>0$ be specified later.
Let the prior $\pi_1$ be the distribution of the following random sparse matrix:
\begin{equation}
	M = t U_I B U_I,
	\label{eq:spM}
\end{equation}
Equivalently, we can define $M$ by
\[
m_{ij} =  \indc{i \in I} \indc{j \in I} u_i u_j b_{ij}.
\]
Therefore the non-zero pattern of $M$ has the desired marginal distribution $\Bern(\frac{k}{p})$, but the entries of $M$ are \emph{dependent}. Alternatively, $M$ can be generated as follows: First choose an $m \times m$ principal submatrix with a uniformly chosen support $I$, fill it with \iid~$\Bern(\frac{k}{m})$ entries, then pre- and post-multiply by a diagonal matrix consisting of independent Rademacher variables, which used to randomize the sign of the leading eigenvector. 
By construction, with high probability, the matrix $M$ is $O(k)$-sparse and, furthermore, its operator norm satisfies $\|M\|_2 \geq k t$.
Furthermore, the corresponding eigenvector is approximately $\mathbf{1}_J$, which is $m$-sparse.

The construction of this prior is based on the following intuition. The operator norm of a matrix highly depends on the correlation of the rows. Given the $\ell_2$-norm of the rows, the largest spectral norm is achieved when all rows are aligned (rank-one), while the smallest spectral norm is achieved when all rows are orthogonal. In the sparse case, aligned support results in large spectral norm while disjoint support in small spectral norm. However, if all rows are aligned, then the signal is prominent enough to be distinguished from noise. Note that a submatrix structure strikes a precise balance between the extremal cases of completely aligned and disjoint support, which enforces that the row support sets are \emph{contained} in a set of cardinality $m$, which is much larger than the row sparsity $k$ but much smaller than the matrix size $p$. In fact, the optimal choice of the submatrix size given by $m \asymp k^2 \wedge \sqrt{kp}$, which matches the structural property given in \prettyref{thm:approxvec}.
The  structure of the least favorable prior, in a way, shows that the optimality of tests based on approximate singular vector is not a coincidence.

Another perspective is that the sparsity constraint on the matrix forces the marginal distribution of each entry in the nonzero pattern $(\indc{M_{ij} \neq 0})$
to be $\Bern(\frac{k}{p})$. However, if all the entries were independent, then it would be very easy to test from noise. 
Indeed, perhaps the most straightforward choice of prior is $M_{ij} \iiddistr t \cdot \Bern(\frac{k}{p})$, where $t \asymp \frac{k}{p}$.
However, the linear test statistic based on $\sum_{ij} M_{ij}$ succeeds unless $\lambda \lesssim 1$. 
We can improve the prior by randomize the eigenvector, i.e., $M_{ij} \iiddistr t u_i u_j \Bern(\frac{k}{p})$, but the $\chi^2$-test in \prettyref{thm:ub.kbig} unless 
$\lambda \lesssim \sqrt{k}$, which still falls short of the desired $\lambda \asymp (kp)^{1/4}$.
Thus, we see that the coupling between the entries is useful to make the mixture distribution closer to the null hypothesis.

	\subsection{Key lemmas}
	\label{sec:keylmm}
	
	The main tool for our lower bound is the $\chi^2$-divergence, defined by
\[
\chi^2(P \, \| \, Q) \triangleq \int \pth{\fracd{P}{Q} - 1}^2 \diff Q,
\]
	which is the variance of the likelihood ratio $\fracd{P}{Q}$ under $Q$. The $\chi^2$-divergence is related to the total variation via the following inequality \cite[p.~1496]{FHT03}:
	\begin{equation}
	\chi^2 \geq \TV \log \frac{1+\TV}{1-\TV}
	\label{eq:chi2tv}
\end{equation}
Therefore the total variation distance cannot goes to one unless the $\chi^2$-divergence diverges. Furthermore, if $\chi^2$-divergence vanishes, then the total variation also  vanishes, which is equivalently to, in view of \prettyref{eq:lecam}, that $P$ cannot be distinguished from $Q$ better than random guessing.

		The following lemma due to Ingster and Suslina (see, \eg, \cite[p.~97]{IS03}) gives a formula for the $\chi^2$-divergence of a normal location mixture with respect to the standard normal distribution.
\begin{lemma}
Let $P$ be an arbitrary distribution on $\reals^m$. Then
\[
\chi^2( N(0,\identity_m) * P \, \| N(0,\identity_m)) = \eexpect{\exp(\Iprod{X}{\tX})}-1
\]
where $*$ denotes convolution and $X$ and $\tX$ are independently drawn from $P$.
	\label{lmm:chi2}
\end{lemma}

	The proof of the lower bound in \prettyref{thm:main} relies on the following lemmas.
	These results give non-asymptotic both necessary and sufficient conditions for certain moment generating functions involving hypergeometric distributions to be bounded, which show up in the $\chi^2$-divergence calculation.
	Let $H \sim \Hyper(p,m,m)$, with $\prob{H =i} = \frac{\binom{m}{i} \binom{p-m}{m-i}}{\binom{p}{m}}, i = \ntok{0}{m}$. 
	
	\begin{lemma}[{\cite[Lemma 1]{CMW13}}]
	\label{lmm:rm}
	Let $p \in \naturals$ and $m \in [p]$. Let $\ntok{B_1}{B_m}$ be independently Rademacher distributed. Denote by
	\[
	G_m \triangleq \sum_{i=1}^m B_i
	\]
the position of a symmetric random walk on $\integers$ starting at 0 after $m$ steps. 
Then there exist an absolute constant $a_0 > 0$ and function $A:(0,a_0) \mapsto \reals_+$ with $A(0+)=0$, such that if $t = \frac{a}{m} \log \frac{\eexp p}{m}$ and $a<a_0$, then
\begin{equation}
\expect{\exp\pth{t G_H^2}} \leq A(a).
	\label{eq:rm}
\end{equation}
	\end{lemma}	

\begin{lemma}[{\cite[Lemma 15, Appendix C]{HWX14}}]
	\label{lmm:H}
	Let $p \in \naturals$ and $m \in [p]$. Then there exist an absolute constant $b_0 > 0$ and function $B:(0,b_0) \mapsto \reals_+$ with $B(0+)=0$, such that if $\lambda = b \pth{\frac{1}{m} \log \frac{\eexp p}{m}  \wedge \frac{p^2}{m^4}}$ and $b<b_0$, then	
\begin{equation}
\expect{\exp\pth{\lambda H^2}} \leq B(b).
	\label{eq:H}
\end{equation}
\end{lemma}

\begin{remark}[Tightness of Lemmas \ref{lmm:rm}--\ref{lmm:H}]
	\label{rmk:tightlmm}
	The purpose of \prettyref{lmm:rm} is to seek the largest $t$, as a function of $p$ and $m$, such that $\expect{\exp\pth{t G_H^2}}$ is upper bounded by a constant non-asymptotically. The condition that $t \asymp \frac{1}{m} \log \frac{\eexp p}{m}$ is in fact both necessary and sufficient. To see the necessity, note that $\prob{G_H = H | H=i} = 2^{-i}$. Therefore 
	\[
	\expect{\exp\pth{t G_H^2}} \geq \expect{\exp(t H^2) 2^{-H}} \geq \exp(t m^2) 2^{-m} \, \prob{H=m} \geq  \exp\pth{tm^2 - m\log \frac{2p}{m}},
	\]
	 which cannot be upper bound bounded by an absolute constant unless $t \lesssim \frac{1}{m} \log \frac{\eexp p}{m}$.
	
Similarly, the condition $\lambda \lesssim \frac{1}{m} \log \frac{\eexp p}{m}  \wedge \frac{p^2}{m^4}$ in 	\prettyref{lmm:H} is also necessary. To see this, note that $\expect{H}=\frac{m^2}{p}$. By Jensen's inequality, we have $\expect{\exp\pth{\lambda H^2}} \geq \exp(\frac{\lambda m^4}{p^2})$. Therefore a necessary condition for \prettyref{eq:rm} is that $\lambda \leq \frac{p^2 \log B}{m^4}$. On the other hand, we have $\expect{\exp\pth{\lambda H^2}} \geq \exp(\lambda m^2 - m\log \frac{p}{m})$, which implies that $\lambda \lesssim \frac{1}{m} \log \frac{\eexp p}{m}$.

\end{remark}

	\subsection{Proof of \prettyref{thm:main}: lower bound}
	\label{sec:pf.lb.main}
	\begin{proof}
	\emph{Step 1}: 	
Fix $t>0$ to be determined later.
Recall the random sparse matrix $M = t U_I B U_I$ defined in \prettyref{eq:spM}, where 
$I$ is chosen uniformly at random from all subsets of $[p]$ of cardinality $k$, 
$u=(\ntok{u_1}{u_p})^\top$ consists of independent Rademacher entries, 
$B$ is a $p\times p$ matrix with \iid~$\Bern(\frac{k}{m})$ entries, and $(u,I,B)$ are independent. 
Equivalently, 
\[
m_{ij} =  \indc{i \in I} \indc{j \in J} u_i u_j b_{ij}.
\]

Next we show that the hypothesis $H_0: X=Z$ versus $H_1: X = M+Z$ cannot be tested with vanishing probability of error, by showing that the $\chi^2$-divergence is bounded. Let $(\tU,\tI,\tB)$ be an independent copy of $(U,I,B)$. Then $\tM=\tU_{\tI} \tB \tU_{\tI}$ is an independent copy of $M$. Put $s = t^2$. By \prettyref{lmm:chi2}, we have
\begin{align}
\chi^2(P_{X|H_0}\,\|\,P_{X|H_1})+1
= & ~ \expect{\exp\pth{\Iprod{M}{\tilde{M}}}} \nonumber \\
= & ~ \expect{\exp\pth{t^2 \Iprod{U_I B U_I}{\tU_{\tI} \tB \tU_{\tI}}}} \nonumber \\
= & ~ \expect{\exp\pth{s \sum_{i\in I\cap \tI} \sum_{j\in I\cap \tI} u_i \tu_i u_j \tu_j b_{ij} \tb_{ij}}} \nonumber \\
\stepa{=} & ~ \expect{\exp\pth{s \sum_{i\in I\cap \tI} \sum_{j\in I\cap \tI} u_i u_j a_{ij}}} \nonumber \\
\stepb{=} & ~ \expect{ \prod_{i \in I\cap \tI} \prod_{j \in J \cap \tJ}  \pth{1+ \frac{k^2}{m^2}(e^{s u_i u_j}-1)}}  \nonumber \\
\stepc{\leq} & ~ \expect{ \exp\sth{\frac{k^2}{m^2} \sum_{i \in I\cap \tI} \sum_{j \in I \cap \tI}  (e^{s u_i u_j}-1)}} \nonumber \\
= & ~ \expect{ \exp\sth{\frac{k^2}{m^2} \sum_{i \in I\cap \tI} \sum_{j \in I \cap \tI}  (u_i u_j \sinh(s) + \cosh(s)-1)}} \nonumber \\
= & ~ \expect{ \exp\sth{\frac{k^2 \sinh(s)}{m^2}  \bpth{\sum_{i \in I\cap \tI} u_i}^2  + \frac{k^2 (\cosh(s)-1)}{m^2}|I \cap \tI|^2  }  }, \label{eq:lb5}
\end{align}
where 
(a) is due to $(\ntok{u_m \tu_m}{u_m \tu_m}) \eqdistr (\ntok{u_1}{u_m})$;
(b) follows from $a_{ij} \triangleq b_{ij} \tb_{ij} \, \iiddistr \, \Bern(\frac{k^2}{m^2})$;
(c) follows from the fact that $\log (1+x) \leq x$ for all $x > -1$;
(d) is because for $b \in \{\pm 1\}$, we have $\exp(s b) = b \sinh(s) + \cosh(s) - 1$.
Recall from \prettyref{lmm:rm} that $\{G_m: m\geq0\}$ denotes the symmetric random walk on $\integers$. Since $I,\tI$ are independently and uniformly drawn from all subsets of $[p]$ of cardinality $k$, we have $H \triangleq |I\cap \tI| \sim \Hyper(p,m,m)$. Define
\begin{align}
A(m,s) \triangleq	 & ~ 	\expect{ \exp\sth{\frac{2k^2 \sinh(s)}{m^2}  G_{H}^2}}, 	\label{eq:Ams}\\
B(m,s) \triangleq	 & ~ 	\expect{ \exp\sth{\frac{2k^2 (\cosh(s)-1)}{m^2} H^2}}. 	\label{eq:Bms}
\end{align}
Applying the Cauchy-Schwartz inequality to the right-hand side of \prettyref{eq:lb5}, we obtain
\begin{equation}
\chi^2(P_{X|H_0}\,\|\,P_{X|H_1})+1 \leq \sqrt{A(m,s) B(m,s)}.
	\label{eq:lbmain}
\end{equation}
Therefore upper bounding the $\chi^2$-divergence boils down to controlling the expectations in \prettyref{eq:Ams} and \prettyref{eq:Bms} separately.

Applying \prettyref{lmm:rm} and \prettyref{lmm:H} to $A(m,s)$ and $B(m,s)$ respectively, we conclude that 
\begin{align}
\frac{k^2 (\cosh(s)-1)}{m^2} \leq  c \pth{\frac{1}{m} \log \frac{\eexp p}{m}  \wedge \frac{p^2}{m^4}}  & ~ \Rightarrow \quad  A(m,s) \leq C
	\label{eq:Ams-s} \\
	\frac{k^2 \sinh(s)}{m^2} \leq  \frac{c}{m} \log \frac{\eexp p}{m}  & ~ \Rightarrow \quad B(m,s) \leq C	\label{eq:Bms-s}
\end{align}
where $c,C$ are constants so that $C\to0$ as $c\to0$. Therefore the best lower bound we get for $s$ is
\begin{equation}
	s^* = \max_{k \leq m \leq p} \sth{ (\cosh-1)^{-1}\pth{\frac{c m}{k^2} \log \frac{\eexp p}{m} \wedge \frac{c p^2}{m^2 k^2}} \wedge  \sinh^{-1}\pth{ \frac{c m}{k^2} \log \frac{\eexp p}{m}  } } \,,
	\label{eq:bests}
\end{equation}
where the inverses $\sinh^{-1}$ and $(\cosh-1)^{-1}$ are defined with the domain restricted to $\reals_+$. 

To simplify the maximization in \prettyref{eq:bests}, we use the following bounds of the hyperbolic functions:
\begin{align}
\sinh^{-1}(y) \geq \log(2 y), ~ (\cosh-1)^{-1}(y) \geq \log y, & \quad y \geq 0. \label{eq:hyper1}
\end{align}
Therefore
\[
	s^* \geq \log \max_{k \leq m \leq p} \pth{ \frac{c m}{k^2} \log \frac{\eexp p}{m} \wedge \frac{c p^2}{m^2 k^2}}.
\]
Choosing $m = \pth{\frac{p^2}{\log p}}^{\frac{1}{3}}$ yields
\begin{equation}
	s^* \gtrsim \log^+ \pth{\frac{p \log p}{k^3}},
	\label{eq:bests-ksmall}
\end{equation}
where $\log^+\triangleq \max\{\log,0\}$.
Note that the above lower bound is vacuous unless $k \leq 	(p \log p)^{\frac{1}{3}}$. To produce a non-trivial lower bound for $k \geq 	(p \log p)^{\frac{1}{3}}$, note that \prettyref{eq:hyper1} can be improved as follows. If the argument $y$ is restricted to the unit interval, then
\begin{align}
\sinh^{-1}(y)  \geq \sinh^{-1}(1) \, y, ~  (\cosh-1)^{-1}(y) \geq \sqrt{y}, & \quad y \in [0,1], \label{eq:hyper0}
\end{align}
which follows from the Taylor expansion of $\cosh$ and the convexity of $\sinh$. Applying \prettyref{eq:hyper0} to \prettyref{eq:bests}, 
\[
	s^* = \max_{m: \frac{c m}{k^2} \log \frac{\eexp p}{m} \leq 1} \pth{\sqrt{\frac{c p^2}{m^2 k^2}} \wedge  \frac{c \sinh^{-1}(1) m}{k^2} \log \frac{\eexp p}{m}}.
\]
Choosing $m = \sqrt{\frac{pk}{ 4c^2 \log \frac{\eexp p}{k}}}$ yields $\frac{cm}{k^2} \log \frac{\eexp p}{k} \leq 1$. We then obtain
\begin{equation}
	s^* \gtrsim \sqrt{\frac{p}{k^3} \log \frac{\eexp p}{k}}.
	\label{eq:bests-kbig}
\end{equation}


\emph{Step 2}: 
We invoke \prettyref{lmm:sep} to conclude $kt$ as a valid lower bound for $\lambda$ with $t=\sqrt{s^*}$ given in \prettyref{eq:bests-ksmall} and \prettyref{eq:bests-kbig}. To this end, we need to show that with high probability, $M$ is $O(k)$-sparse and $\|M\|_2 = \Omega(k t)$.
Define events
\[
E_1 = \{M \in \calM(p,2k) \}, \quad E_2 = \{\opnorm{M} \geq kt/2 \}.
\]
It remains to show that both are high-probability events.
Since $I$ is independent of $B$, we shall assume, without loss of generality, that $I=[m]$.
For the event $E_1$, by the union bound and Hoeffding's inequality, we have
\begin{equation}
\prob{\comp{E_1}} = \prob{B_{II} \notin \Theta(m,2k)} 
	\leq m^2 \prob{\sum_{i=1}^m b_{i1} \geq 2k} \leq m^2 \exp(-mk^2) = o(1),
	\label{eq:E1p}
\end{equation}
where $b_{i1} \iiddistr \Bern(\frac{k}{m})$.
For the event $E_2$, again by Hoeffding's inequality,
\begin{align}
\prob{E_2}
= & ~ \prob{\|B_{II}\|_2 \geq k/2} \geq \prob{\|M \mathbf{1}_I\|_2 \geq \frac{k}{2} \|\mathbf{1}_I\|_2}	\nonumber \\
\geq & ~ \prob{\sum_{j=1}^m b_{ij} \geq \frac{k}{2}, \forall i \in [m]}	\geq 1 - m \, \prob{\sum_{j=1}^m b_{1j} < \frac{k}{2}}	\\
\geq & ~ 1 - m \exp(-mk^2/4)	= 1 - o(1) \label{eq:E22}.
\end{align}
The desired lower bound now follows from \prettyref{lmm:sep}.

Finally, we note that the lower bound continues to hold up to constant factors even if $M$ is constrained to be symmetric. Indeed, 
we can replace $M$ with the symmetrized version $M'=[\begin{smallmatrix} 0  & M \\ M^\top & 0\end{smallmatrix}]$ and note that the bound on $\chi^2$-divergence remains valid since $\Iprod{M'}{\tM'}=2\Iprod{M}{\tM}$.
\end{proof}

\section{Detecting sparse covariance matrices}
\label{sec:cov}

\subsection{Test procedures and upper bounds}
	\label{sec:cov-ub}

Let $X_1,\ldots,X_n$ be independently sampled from $N(0,\Sigma)$. 
Define the sample covariance matrix as
\begin{equation}
S = \frac{1}{n} \sum_{i=1}^n X_i X_i^\top,
\label{eq:samplecov}
\end{equation}
which is a sufficient statistic for $\Sigma$.

The following result is the counterpart of \prettyref{thm:ub.ksmall} for entrywise thresholding that is optimal in the highly sparse regime:
\begin{theorem}
	\label{thm:ub.ksmall-cov}
	Let $C,C'$ be constants that only depend on $\tau$.
	Let $\epsilon \in (0,1)$. Define $\hat \Sigma = (S_{ij} \indc{|S_{ij}| \geq t})$, 
	where $\tau  = \sqrt{C\log \frac{p}{\epsilon}}$. Assume that
	$n \geq C' \log p$. If
	$\lambda \sqrt{n} > 2 k t$, then the test $
\psi(S) = \indc{\|\hat \Sigma\|_2 \geq \lambda}$
satisfies
\[
\Prob_{\identity}(\psi=1) + \sup_{\Sigma \in \Xi(p,k,\lambda,\tau)}  \Prob_{\Sigma}(\psi=0)   \leq \epsilon
\]
for all $1 \leq k \leq p$. 
\end{theorem}

To extend the test \prettyref{eq:wtest} to covariance model, 
we need a test statistic for $\fnorm{\Sigma-\identity}^2$.
Consider the following U-statistic proposed in \cite{chen2010tests,CM13}:
\begin{equation}
Q(S) = p + \frac{1}{\binom{n}{2}} \sum_{1 \leq i < j \leq n} \iprod{X_i}{X_j}^2 - \iprod{X_i}{X_i} - \iprod{X_j}{X_j},
\label{eq:QS}
\end{equation}
Then $Q(S)$ is a unbiased estimator of $\fnorm{\Sigma-\identity}^2$.
We have the following result for the moderately sparse regime:

\begin{theorem}
	\label{thm:ub.kbig-cov}
Let 
\[
m = C \sqrt{\frac{kp}{\log \frac{\eexp p}{k}}}.
\]
where $C$ is the universal constant from \prettyref{thm:approxvec}. Define the following test statistic
\begin{equation}
	T_m(S) = \max\{\opnorm{S_{II}}: I \subset [p], |I|=m\}
	\label{eq:TmX-cov}
\end{equation}
and the test
\begin{align}
\psi(S) = \indc{Q(S) \geq s} \vee \indc{T_m(X) \geq t}
\label{eq:wtest-cov} 
\end{align}
where
\begin{equation}
s \triangleq 2 \log \frac{1}{\epsilon} + 2 p \sqrt{\log \frac{1}{\epsilon}}, \quad t \triangleq 2\sqrt{m} + 4 \sqrt{m\log \frac{ep}{m}}.
\label{eq:st-cov}
\end{equation}
There exists a universal constant $C_0$ such that the following holds.
For any $\epsilon \in (0,1/2)$, if
	\begin{equation}
	\lambda \geq C_0 \sth{kp \log \frac{1}{\epsilon} \log \pth{\frac{p}{k} \log \frac{1}{\epsilon}  }}^{\frac{1}{4}},
	\label{eq:ub.kbig-cov}
\end{equation}
then the test \prettyref{eq:wtest-cov} satisfies
\[
\Prob_0(\psi=1) + \sup_{M \in \Theta(p,k,\lambda)}  \Prob_{M}(\psi=0)   \leq \epsilon
\]
holds for all $1 \leq k \leq p$.
\end{theorem}

The proofs of Theorems \ref{thm:ub.ksmall-cov} and \ref{thm:ub.kbig-cov}  parallel those of Theorems \ref{thm:ub.ksmall} and \ref{thm:ub.kbig}. Next we point out the main distinction.
For \prettyref{thm:ub.ksmall-cov}, the only difference is the Gaussian tail is replaced by the concentration inequality 
$\prob{|S_{ij}-\Sigma_{ij}| \geq a} \leq c_0 \exp(-c_1 nt^2)$ for all $|t| \leq c_2$, where $c_i$'s are constants depending only on $\tau$ \cite[Eq.~(26)]{CZ12}.
For \prettyref{thm:ub.kbig-cov}, 
let $\tS \triangleq\Sigma^{-\frac{1}{2}} S \Sigma^{-\frac{1}{2}}$, which is a $k\times k$ standard Wishart matrix with $n$ degrees of freedom. Applying the deviation inequality in \cite[Proposition 4]{CMW12}, we have $\Expect[\Opnorm{\tS-I_k}^2] \lesssim \frac{k}{n}+\frac{k^2}{n^2}$. Since $\opnorm{S-\Sigma} \leq \opnorm{\Sigma} \Opnorm{\tS-I_k}$, we have $\expect{\opnorm{S-\Sigma}^2} \lesssim \lambda^2 \pth{\frac{k}{n}+\frac{k^2}{n^2}}$. 

\subsection{Proof of the lower bound}
	\label{sec:cov-lb}
	In this subsection we prove the lower bound part of \prettyref{thm:main-cov}.	
		We begin by stating a counterpart of \prettyref{lmm:chi2} for covariance model, which also gives an inequality relating the $\chi^2$-divergences of the mean model and the covariance model:
\begin{lemma}
Let $\lambda \geq 0$. Let $Q$ be the distribution of a $p \times p$ symmetric random matrix $T$ such that $\opnorm{T} \leq \lambda$ almost surely. 
Let $\pi_Q^n \triangleq \int N(0, \lambda \identity_p + T)^{\otimes n} Q(\diff T)$ denote the $n$-sample scale mixture. Then
\begin{align}
\chi^2( \pi_Q^n \, \| N(0, \lambda \identity_p)^{\otimes n})
= & ~ \expect{\det(\identity_p - \lambda^{-2} T \tilde{T})^{-\frac{n}{2}}}-1 \\
\geq & ~ \expect{\exp\pth{\frac{n}{2\lambda^{2}}\Iprod{T}{\tilde{T}} }}-1.	\label{eq:chi2.lb}
\end{align}
where $T$ and $\tilde{T}$ are independently drawn from $Q$.

Furthermore, if $\opnorm{T} \leq \delta \lambda$ almost surely for some $\delta \in (0,1)$, then
\begin{align}
\chi^2( \pi_Q^n \, \| N(0, \lambda \identity_p)^{\otimes n}) 
\leq & ~ \expect{\exp\pth{\frac{n}{\lambda^{2}}\Iprod{T}{\tilde{T}}}}^{\frac{1}{2}} \expect{\exp\pth{\frac{n}{(1-\delta^2) \lambda^{4}} \fnorm{T \tilde{T}}^2}}^{\frac{1}{2}}-1.
\label{eq:chi2.ub}
\end{align}
	\label{lmm:chi2.cov}
\end{lemma}
\begin{proof}
	Let $g_{i}$ be the density function of $N\left(0,\Sigma _{i}\right) $ for $i=0,1$ and $2$, respectively. Then 
\begin{align}
\int \frac{g_{1}g_{2}}{g_{0}}
= & ~ \det(\Sigma_0)^{-1} \left[ \det \left( \Sigma_{0}(\Sigma _{1}+\Sigma _{2}) - \Sigma_1\Sigma_2\right)\right] ^{-\frac{1}{2}} \label{eq:g012}
\end{align}
provided that 
\begin{equation}
\Sigma_{0}(\Sigma _{1}+\Sigma _{2}) \succeq \Sigma_1\Sigma_2
\label{eq:g012cond}
\end{equation}
 otherwise, the integral on the right-hand side of \prettyref{eq:g012} is infinite.
Conditioning on two independent copies $T,\tilde{T}$ and applying \prettyref{eq:g012} with $\Sigma_0=\lambda \identity_p$, $\Sigma_1=\lambda \identity_p + T$ and $\Sigma_2=\lambda \identity_p + \tilde{T}$,  which satisfies \prettyref{eq:g012cond}, we obtain
\begin{align*}
\chi^2( \pi_Q^n \, \| N(0, \lambda \identity_p)^{\otimes n}) + 1
= & ~ \Expect_{T,\tilde{T}}\qth{ \pth{\int \frac{g_{1}g_{2}}{g_{0}}}^n} \\
= & ~ \expect{\det(\identity_p - \lambda^{-2} T \tilde{T})^{-\frac{n}{2}}} \\
= & ~ \expect{\exp\pth{-\frac{n}{2} \log \det(\identity_p - \lambda^{-2} T \tilde{T})}} \\
\stepa{\geq} & ~ \expect{\exp\pth{\frac{n}{2} \lambda^{-2} \Tr(T \tilde{T})}} \\
= & ~ \expect{\exp\pth{\frac{n}{2\lambda^{2}}  \iprod{T}{\tilde{T}}}} \\
= & ~ \chi^2( Q * N(0, 2\lambda^{2}n^{-1} \identity_p) \, \| N(0, 2\lambda^{2}n^{-1}  \identity_p)) + 1
\end{align*}
where (a) is due to $\log \det(I+A) \leq \Tr(A)$.

If $\lambda_{\max}(T) \leq \delta \lambda < \lambda$, then $\lambda_{\max}(\lambda^{-2} T \tilde{T}) < \delta^2$. Using 
\[
\log(1-\lambda) \geq -\frac{\lambda}{1-\lambda} = -\lambda - \frac{\lambda^2}{1-\lambda} \geq -\lambda - \frac{\lambda^2}{1-\delta^2},
\]
for all $\lambda \leq \delta^2 < 1$, we have
\[
\log \det(I - \lambda^{-2} T \tilde{T}) \geq -\lambda^{-2} \Tr(T \tilde{T}) - \frac{1}{(1-\delta^2) \lambda^{4}} \fnorm{T \tilde{T}}^2,
\]
which gives 
\[
\chi^2( \pi_Q^n \, \| N(0, \lambda \identity_p)^{\otimes n}) 
\leq \expect{\exp\pth{\frac{n}{2\lambda^{2}}\Iprod{T}{\tilde{T}} +  \frac{n}{2 (1-\delta^2) \lambda^{4}} \fnorm{T \tilde{T}}^2}}-1,
\]
 and, upon applying Cauchy-Schwarz, \prettyref{eq:chi2.ub}.
\end{proof}

Next we apply \prettyref{lmm:chi2.cov} to obtain minimax lower bound for testing sparse covariance matrices as defined in \prettyref{eq:htsparse-cov}.
Throughout the remainder of this subsection, $c,c',c_0,\cdots$ denote absolute constants whose value might vary at each occurrence.
Recall that $M$ is the $p\times p$ random matrix defined in \prettyref{eq:spM}, with
$t = \sqrt{\frac{s}{n}}$,
and $s>0$ and $k \leq m \leq p$ are to be specified later.

Define the event
\begin{align}
E_3 = & ~  \sth{\opnorm{M} \leq \frac{1}{2} } \label{eq:E3} \\
E_4 = & ~  \sth{\opnorm{B_{II} - p \allones_{II}} \leq \sqrt{c_0 \pth{k \vee \frac{\log m}{\log \frac{e \log m}{k}}}} } \label{eq:E4} 
\end{align}
and set $E^{\star} = E_3 \cap E_4$.
Next we show that 
\begin{equation}
\prob{E^\star} \geq 1 - o(1).
\label{eq:PE3}
\end{equation}
For $E_3$,
	note that $\norm{M}_2 = t \norm{B_{II}}$. Since $I$ is independent of $B$, 
	we shall assume that $I=[m]$ and let $B'=B_{II}$. Since 
	$\|B_{II}\|_2 \leq \|B_{II}\|_1 \|B_{II}\|_\infty$, 
	similar to \prettyref{eq:E22}, Hoeffding inequality implies that $\prob{\|B_{II}\|_1 \geq 2k} = \prob{\|B_{II}\|_\infty \geq 2k} \leq m^2\exp(-mk^2/4)$. Therefore
	$\|M\|_2 \leq 2tk $ with probability at least $1-2m^2\exp(-mk^2/4)$.
For $E_4$, it follows\footnote{The result in \cite[Theorem 5]{HWX14b} deals with symmetric matrices. Here, since $B_{II}$ is an $m\times m$ matrix consisting of iid $\Bern(k/m)$ entries, the result follows from combining \cite[(15)]{HWX14b} and Talagrand's concentration inequality at the end of the proof therein.} from \cite[Theorem 5]{HWX14b} that 
\[
\opnorm{B_{II} - \Expect[B_{II}]} \leq \sqrt{c_0 \pth{k \vee \frac{\log m}{\log \frac{e \log m}{k}}}},
\]
with probability at least $1-\exp(-c_1(k \vee \frac{\log m}{\log \frac{e \log m}{k}}))$. Thus $\prob{E_4} = 1-o(1)$.

Consider $\Sigma=\identity_{2p} + T$, where 
\[
T = \begin{bmatrix} 0  & M \\ M^\top & 0\end{bmatrix}.
\]
Note that $TT^\top=\begin{bmatrix} MM^\top   & 0 \\ 0& M^\top M\end{bmatrix}$ and hence
$\|T\|_2 = \norm{M}_2$.
Let $Q$ denote the law of $T$ conditioned on the event $E^{\star}$, and 
By \prettyref{eq:E1p} and \prettyref{eq:E22}, we have
$\Sigma \in \Xi(2p,2k,kt/2,3/2)$ with probability tending to one. 
Thus it remains to bound the $\chi^2$-divergence.

Let $\pi_Q^n$ denote the mixture of $N(0,\identity_{2p}+T)^{\otimes n}$ induced by the prior $Q$. 
Let $\tM$, $\tS$ and $\tE^{\star}$ are independent copies of $M$, $T$, and $E^{\star}$, respectively.
Applying \prettyref{lmm:chi2.cov} with $\lambda=1$ and $\delta=\frac{1}{2}$, we have
\begin{align*}
& ~  \chi^2( \pi_Q^n \, \| N(0, \identity_{2p})^{\otimes n}) \\
\leq & ~  \expect{\exp\pth{n\Iprod{T}{\tilde{T}}} \Big| E^{\star},\tilde E^{\star}  }^{\frac{1}{2}} \expect{\exp\pth{2n\fnorm{T \tilde{T}}^2}  \Big| E^{\star},\tilde E^{\star}  }^{\frac{1}{2}}-1\\
= & ~  \expect{\exp\pth{2n\Iprod{M}{\tilde{M}}} \Big| E^{\star},\tilde E^{\star}  }^{\frac{1}{2}} \expect{\exp\pth{2n (\fnorm{M \tilde{M}^\top}^2 + \fnorm{M^\top \tilde{M}}^2)}  \Big| E^{\star},\tilde E^{\star}  }^{\frac{1}{2}}-1\\
\stepa{\leq} 
& ~  \expect{\exp\pth{2n\Iprod{M}{\tilde{M}}} \Big| E^{\star},\tilde E^{\star}  }^{\frac{1}{2}} 
\expect{\exp\pth{4n \fnorm{M \tilde{M}}^2 }  \Big| E^{\star},\tilde E^{\star}  }^{\frac{1}{2}}
-1\\
\leq
& ~ \frac{1}{\prob{E^{\star}}^2} \expect{\exp\pth{2n\Iprod{M}{\tilde{M}}}  }^{\frac{1}{2}} \expect{\exp\pth{4n \fnorm{M \tilde{M}}^2 }  }^{\frac{1}{2}} -1
\end{align*}
where (a) follows from the Cauchy-Schwarz inequality and the fact that $M$ and $M^\top$ have the same distribution.

In \prettyref{sec:pf.lb.main}
we have already shown that $\expect{\exp\pth{2n\Iprod{M}{\tilde{M}}}  }$ is bounded by a constant, provided that 
\prettyref{eq:Ams-s} and \prettyref{eq:Bms-s} holds.
 To complete the proof of the lower bound, it remains to show  that
\begin{equation}
\expect{\exp\pth{4 n \fnorm{M \tilde{M}}^2}}
\label{eq:covgoal}
\end{equation}
is bounded under the condition of 
\prettyref{eq:nassumption-basic} and \prettyref{eq:nassumption-extra}.
To this end, recall from \prettyref{eq:spM} that $M = t H$ and $\tM = t \tH$, where
$H =U_I B U_I$ and $\tH=\tU_\tI \tB \tU_\tI$. 
Note that $\Expect[B]=\epsilon \allones$, where $\epsilon = \frac{k}{m}$.
Write
\[
H = U_I B U_I = U_I (B-\epsilon \allones) U_I  + \epsilon U_I \allones U_I  
= U_I (B-\epsilon \allones) U_I  + \epsilon v v^\top
\]
where $v=U_I \ones$ is supported on $I$ with $m$ independent Rademacher non-zeros.
Then
\[
H\tH
= 
\underbrace{\epsilon^2 v v^\top  \tv \tv^\top}_{h_1} +
\underbrace{\epsilon U_I (B-\epsilon \allones) U_I \tv \tv^\top}_{h_2}
+ \underbrace{\epsilon v v^\top \tU_\tI (\tB-\epsilon \allones) \tU_\tI }_{h_3}
+ \underbrace{U_I (B-\epsilon \allones) U_I  \tU_\tI (\tB-\epsilon \allones) \tU_\tI  }_{h_4}
\]
where the first three terms are rank-one matrices.
Since
$\fnorm{H\tH}^2 \leq 4 \sum_{i=1}^4 \fnorm{h_1}^2$, H\"older's inequality implies 
\[
\expect{\exp\pth{4 n \fnorm{M \tilde{M}}^2}}
\leq 
\prod_{i=1}^4 \qth{\expect{\exp\pth{64 n t^4 \fnorm{h_i}^2}}}^{1/4}
\]
%
%
%
	We proceed to bound the four terms separately. First note that $h_1 = \epsilon^2 \Iprod{v}{\tv} v \tv^\top$ and hence
	$\fnorm{h_1}^2 = \epsilon^4 \Iprod{v}{\tv}^2 \|v\|_2^2\|\tv\|_2^2 = \epsilon^4 m^2 \Iprod{v}{\tv}^2 $.
Note that 
$\Iprod{v}{\tv} \eqdistr G_H$, where $G_H$, defined in \prettyref{lmm:rm}, is the sum of $\Hyper(p,m,m)$ number of independent Rademacher random variables.
In view of \prettyref{lmm:rm}, we have
$\expect{\exp\pth{64 n t^4 \fnorm{h_1}^2}} \lesssim 1$,
 provided that
\begin{equation}
n t^4\epsilon^4 m^2 = \frac{k^4 s^2}{m^2 n}\lesssim \frac{1}{m} \log \frac{\eexp p}{m}.
\label{eq:cov1}
\end{equation}

Next we bound $h_2$ and $h_3$, which have the same distribution.
Note that $h_2=\epsilon U_I (B-\epsilon \allones) U_I \tv \tv^\top$ and hence 
$\fnorm{h_2}^2 =\epsilon^2 \|U_I (B-\epsilon \allones) U_I \tv\|_2^2  \|\tv\|_2 =  \epsilon^2m  \| \mathbf{1}_I (B-\epsilon \allones) U_I \tv\|_2^2$.
Therefore
\begin{align}
\expect{\exp\pth{64 n t^4 \fnorm{h_2}^2}}
=  & ~ 	\expect{\exp\pth{64 n t^4 \epsilon^2m  \fnorm{ \mathbf{1}_I (B-\epsilon \allones) U_I \tv}^2  }} \nonumber \\
= & ~ 		\expect{\exp\pth{64 n t^4 \epsilon^2m  \fnorm{(B-\epsilon \allones)_{I,I\cap \tI} }^2  }}, \nonumber \\
= & ~ 		\expect{\expect{\exp\pth{\tau (S-\epsilon L)^2 }|L}^m}, \nonumber
\end{align}
where $L \triangleq |I\cap \tI|$, $S\sim \Binom(L,\epsilon)$, and  
\[
\tau
\triangleq 64 n t^4 \epsilon^2m
= 64 \frac{k^2 s^2}{m n}.
\]
Assume that
\begin{equation}
\tau \lesssim \frac{1}{m}.
\label{eq:cov2}
\end{equation}
Recall that for any $\epsilon$, $\Bern(\epsilon)$ is subgaussian with parameter at most a constant $c$. Therefore $S$ is subgaussian with parameter at most $cL$. 
By the equivalent characterization of subgaussian random variables \cite[Proposition 2.5.2]{vershynin2016high}, we have
\[
\expect{\exp\pth{\tau (S-\epsilon L)^2 }|L}
\leq 
\exp\pth{c \tau^2 L}
\]
provided that $\tau^2 L \leq c'$. 
Therefore
\begin{align}
\expect{\exp\pth{64 n t^4 \fnorm{h_2}^2}}
\leq  & ~ 	\expect{\exp\pth{c \tau^2 m L }} \nonumber \\ 
\stepa{\leq}  & ~ 	\pth{1+\frac{m}{p}(\exp\pth{c \tau^2 m  }-1)}^m \nonumber \\ 
\stepb{\leq}  & ~ 	\exp\pth{c\frac{\tau^2 m^3}{p}} \stepc{\lesssim} 1 \nonumber
\end{align}
where (a) follows from the fact that hypergeometric distribution is stochastically dominated by binomial in the convex ordering \cite{Hoeffding63};
(b) and (c) follow from \prettyref{eq:cov2} and hence $\tau^2 m \leq 1$ and $\frac{\tau^2 m^3}{p} \leq \tau^2m^2 \lesssim 1$.

Finally, we deal with $h_4$, which is the term that requires the extra condition \prettyref{eq:nassumption-extra} on the sample size.
	Note that 
	$\rank(h_4) \leq |I \cap \tI|$ and that $\|h_4\|_2 \leq \norm{B_{II}-\epsilon \allones_{II}}_2  \|\tB_{\tI\tI}-\epsilon \allones_{\tI\tI}\|_2 $. 
	In view of the event $E^\star$ we have conditioned on, 
	Therefore 
		\[
		\fnorm{h_4}^2 \leq \norm{B_{II}-\epsilon \allones_{II}}_2^2  \|\tB_{\tI\tI}-\epsilon \allones_{\tI\tI}\|_2^2  |I \cap \tI| \leq c \rho^2 |I \cap \tI| ,
		\]
		where $\rho \triangleq k \vee \frac{\log m}{\log \frac{e \log m}{k}}$. Hence 
\begin{align}
\expect{\exp\pth{64 n t^4 \fnorm{h_4}^2}}
\leq & ~ \expect{\exp\pth{c n t^4 \rho^2 |I \cap \tI|}}	= \expect{\exp\pth{\frac{c \rho^2 s^2}{n} |I \cap \tI|}}	\nonumber \\
\leq & ~ \pth{1+\frac{m}{p}\pth{\exp\pth{\frac{c \rho^2 s^2}{n}}  -1}}^m \nonumber \\
\leq  & ~ \exp\pth{\frac{c m^2 \rho^2 s^2}{np}  } \lesssim 1, \nonumber 
\end{align}
provided that 
\begin{equation}
\frac{\rho^2 s^2}{n}  \lesssim 1
\label{eq:cov3}
\end{equation}
and 
\begin{equation}
\frac{m^2 \rho^2 s^2}{np}  \lesssim 1.
\label{eq:cov4}
\end{equation}

To finish the proof, we need to choose the parameters to ensure that 
that \prettyref{eq:Ams-s}, \prettyref{eq:Bms-s}, \prettyref{eq:cov1}--\prettyref{eq:cov4} hold simultaneously. 
Let
\begin{align*}
s = & ~ \delta \log \frac{p}{c k^3}, \qquad m= k^2 \pth{\frac{p}{k^3}}^{\delta},   &  \text{if } k \leq p^{1/3} \\
s = & ~\sqrt{\frac{cp}{k^3} \log \frac{ep}{k}}, \qquad m= \sqrt{\frac{cpk}{\log \frac{ep}{k}}},   &  \text{if } k \geq p^{1/3} 
\end{align*}
so that \prettyref{eq:Ams-s} and \prettyref{eq:Bms-s} hold; here $\delta$ is any constant in $(0, \frac{2}{3}]$.
Moreover, the basic assumption on the sample size 
\begin{equation}
n \gtrsim k^2 s \log^2 p
\label{eq:nbasic}
\end{equation}
guarantees \prettyref{eq:cov1}, \prettyref{eq:cov2} and \prettyref{eq:cov3}.
Finally, the extra assumption on the sample size that 
\begin{equation}
n \gtrsim 
\frac{m^2 k^2 s^2}{p} 
= \begin{cases}
\frac{k^6}{p}  \pth{\frac{p}{k^3}}^{2 \delta} \log^2 p &  k\leq p^{1/3}\\
p & k\geq p^{1/3}\\
\end{cases}
\label{eq:nextra}
\end{equation}
ensures \prettyref{eq:cov4} hold simultaneously. 
The lower bound $k \sqrt{\frac{s}{n}}$  follows from \prettyref{lmm:sep} and \prettyref{lmm:chi2.cov}, completing the proof.

\section{Computational limits}
	\label{sec:computational}

In this section we address the computational aspects of detecting sparse matrices in both the Gaussian noise and the covariance model.

\paragraph{Gaussian noise model}
The computational hardness of the red region (reducibility from planted clique) in \prettyref{fig:phase} follows from that of submatrix detection in Gaussian noise \cite{BI12,MW13b}, which is a special case of the model considered here. 
	The statistical and computational boundary of submatrix detection is shown in \prettyref{fig:phase22}, in terms of the tradeoff between the sparsity $k=p^\alpha$ and the spectral norm of the signal $\lambda=p^{\beta}$.
	Below we explain how \prettyref{fig:phase22} follows from the results in \cite{MW13b}.

The setting in \cite{MW13b} also deals with the additive Gaussian noise model \prettyref{eq:GLM}, where, under the alternative, the entries of the mean matrix $M$ is at least $\theta$ on a $k\times k$ submatrix and zero elsewhere, with $k=p^{\alpha}$ and $\theta=p^{-\gamma}$. 
Since $\|M\|_2 \geq k\theta$, this instance is included in the alternative hypothesis in \prettyref{eq:htsparse} with $\lambda=p^{\beta}$ and $\beta =\alpha-\gamma$.
It is shown that (see \cite[Theorem 2 and Fig.~1]{MW13b})
detection is computationally at least as hard as solving the planted clique problem when $\gamma > 0\vee(2\alpha-1) $, i.e., $\beta < \alpha \wedge (1-\alpha)$. 
Note that this bound is not monotone in $\alpha$, which can be readily improved to $\beta < \alpha \wedge \frac{1}{2}$, corresponding to the computational limit in \prettyref{fig:phase22}. 
Similarly, detection is statistically impossible when $\gamma > \frac{\alpha}{2}\vee(2\alpha-1)$, i.e., $\beta < \frac{\alpha}{2} \wedge (1-\alpha)$. 
Taking the monotone upper envelope leads to $\beta < \frac{\alpha}{2} \wedge \frac{1}{3}$, yielding the statistical limit in \prettyref{fig:phase21}.
Finally, 
\prettyref{fig:phase} can be obtained by 
superimposing the statistical-computational limits in \prettyref{fig:phase21} on top of the statistical limit obtained in the present paper as plotted in \prettyref{fig:phase22}.

\begin{figure}[ht]
	\centering
	\subfigure[Statistical boundary for detecting sparse matrices (this paper).]%
	{\label{fig:phase21} 
\begin{tikzpicture}[domain=0:1,xscale=4.5,yscale=4.5, thick]
\draw[->] (0,0) node [below] {\scriptsize 0} -- (0,1/3) node[left] {\scriptsize $\frac{1}{3}$}  -- (0,1.1) node[above] {\scriptsize $\beta$};
\draw[->] (0,0) -- (1/3,0) node[below] {\scriptsize $\frac{1}{3}$} -- (1,0) node[below] {\scriptsize 1} -- (1.1,0) node[right] {\scriptsize $\alpha$};
\filldraw[fill=black!50!white, draw=black] (0,0) -- (1/3,1/3) -- (1,1/2)--(1,0) -- cycle; 
\draw (0,1) -- (1,1)--(1,0);
\node at (0.8,.1) [left,rotate=0,align=center] {impossible};
\node at (0.3,0.7) [right,align=center] {possible};
\end{tikzpicture}
	}
	\subfigure[Statistical-computational boundary for detecting submatrices  \cite{MW13b}.]%
	{\label{fig:phase22}
	\begin{tikzpicture}[domain=0:1,xscale=4.5,yscale=4.5, thick]
\draw[->] (0,0) node [below] {\scriptsize 0} -- (0,1/3) node[left] {\scriptsize $\frac{1}{3}$}-- (0,1/2) node[left] {\scriptsize $\frac{1}{2}$}  -- (0,1) node[left] {\scriptsize 1} -- (0,1.1) node[above] {\scriptsize $\beta$};
\draw[->] (0,0) -- (1/2,0) node[below] {\scriptsize $\frac{1}{2}$} -- (2/3,0) node[below] {\scriptsize $\frac{2}{3}$} -- (1,0) node[below] {\scriptsize 1} -- (1.1,0) node[right] {\scriptsize $\alpha$};
\filldraw[fill=green!30!white, draw=black] (0,0) -- (1/2,1/2) -- (1,0)--(1,1)--(0,1) -- cycle; 
\filldraw[fill=black!50!white, draw=black] (0,0) -- (2/3,1/3) -- (1,1/3)--(1,0) -- cycle; 
\filldraw[fill=black!20!white, draw=black] (0,0)--(2/3,1/3)--(1,1/3)--(1,1/2) -- (1/2,1/2) -- cycle; 
\node at (0.8,.1) [left,rotate=0,align=center] {impossible};
\node at (0.5,0.8) [right,align=center] {easy};
\draw (.32,.28) node [below,align=center,rotate=30] {PC hard};
\end{tikzpicture}
}
	\caption{Detection boundary for $k$-sparse matrices and $k\times k$ submatrices $M$ in noise, where $k=p^\alpha$ and $\|M\|_2=\lambda=p^\beta$.}%
\label{fig:phase2}%
\end{figure}
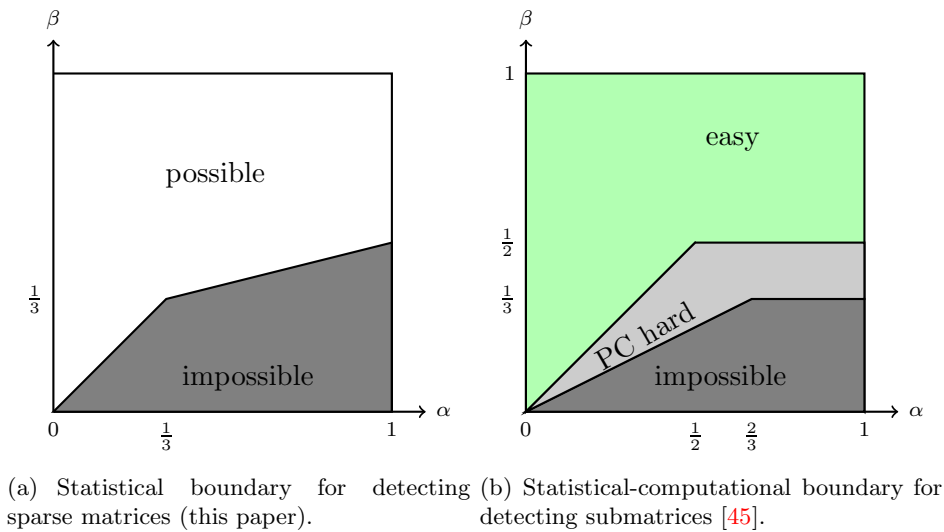

\paragraph{Sparse covariance model}
For the problem of detecting sparse covariance matrices, which is defined by the 4-tuple $(n,p,k,\lambda)$, the picture is less complete than the additive-noise counterpart; this is mainly due to the extra parameter $n$. Indeed, the statistical lower bound in \prettyref{thm:main-cov} holds under the extra assumptions \prettyref{eq:nassumption-basic} and \prettyref{eq:nassumption-extra} that the sample size is sufficiently large, while the current computational lower bound for sparse PCA in the literature \cite{BI12,wang2016statistical,gao2017sparse} also requires a number of conditions including the assumption of $n \leq p$.
Nevertheless, if we still let $k=p^{\alpha}$ and $\lambda \sqrt{n} = p^{\beta}$ and focus on the tradeoff between the $(\alpha,\beta)$ pair, 
the statistical and computational limits in \prettyref{fig:phase} continue to hold. Next we explain how to deduce the computational hardness of the red region 
 from that of sparse PCA in the spiked Gaussian covariance model \cite{gao2017sparse}.

To this end, due to monotonicity, it suffices to demonstrate a ``hard instance'', i.e., a sequence of triples $(n,\lambda,k)$ indexed by $p$, for every $(\alpha,\beta)$ such that $\frac{1}{3}<\alpha<\frac{1}{2}$ and $\beta<1$.
Given samples $X_1,\ldots,X_n \iiddistr N(0,\Sigma)$, the computational aspect of testing 
\begin{equation}
H_0: \Sigma=\identity, \quad \text{versus} \quad H_1: \Sigma=\identity+\lambda uu^\top,
\label{eq:spca}
\end{equation}
 where the eigenvector $u$ is both $k$-sparse and unit-norm, has been studied in \cite{gao2017sparse}. 
Fix $\alpha \in (\frac{1}{3},\frac{1}{2})$. 
Let $n=p^{\eta}$, $k=p^\alpha$ and $\lambda = \frac{c k^2}{n \log^2 n}$, so that $\beta=2\alpha-\eta$, and let $\frac{1}{a}\leq \eta \leq 1$ to be chosen later; here $a>1$ and $c>0$ are absolute constants from \cite[Theorem 5.4]{gao2017sparse}.
By assumption, $(2\alpha,4\alpha) \cap (\frac{1}{a},1) \neq \emptyset$; pick any $\eta$ therein. 
Then we have $\lambda \ll 1$ and \prettyref{eq:spca} is indeed an instance of \prettyref{eq:htsparse1}. 
By the choice of the parameters, the conditions of \cite[Theorem 5.4]{gao2017sparse} are fulfilled, namely, $\beta<\alpha$ and $\alpha > \frac{\eta}{4}$, and the detection problem \prettyref{eq:spca} and hence \prettyref{eq:htsparse1} are at least as hard as the planted clique problem.


\section{Discussions}
	\label{sec:discuss}
	
In this paper, we studied the fundamental limits for sparse matrix detection from both the statistical and computational perspectives, where the alternative hypothesis is defined in terms of the spectral norm. The sparse matrices considered here have no apparent combinatorial structure and the corresponding estimation problem has no computational issue at all, but the detection problem has a surprising computational barrier when the sparsity level exceeds the cubic root of the matrix size. In this section we discuss two related problems, one is the detection problem when the alternative hypothesis is defined in terms of the Frobenius norm and another is the localization and estimation of a sparse matrix.

	\subsection{Alternative hypothesis defined by the Frobenius norm}
	\label{sec:altf}

As opposed to the alternative hypothesis in \prettyref{eq:htsparse} for $k$-sparse matrices defined by the spectral norm, one can consider the detection problem with the alternative hypothesis defined in terms of the Frobenius norm:
	\begin{equation}
	\begin{cases}
	H_0: & M = 0 \\
H_1: & \fnorm{M} \geq \lambda, M\in \Theta(p,k,\lambda).
	\end{cases}
	\label{eq:htsparsef}
\end{equation}
It turns out that in this case the sparsity \emph{plays no role} in improving the detection boundary, in the sense that the optimal separation scales as $\lambda^*(k,p) \asymp \sqrt{p}$ for all $k \geq 1$. 

The intuition behind this result is the well-known fact that in the Gaussian sequence model, the sparsity of the signal does not help in the so-called ``dense regime'' when the sparsity level exceeds the square-root of the dimension \cite{Ingster97,DJ.HC}. Here for $k$-sparse $p\times p$ matrices in the sense of \prettyref{def:sparsemat}, the number of nonzeros can be as large as $kp$ (e.g., block diagonal consisting of $p/k$ number of $k\times k$ blocks), which, since the ambient dimension is $p^2$, lies in the dense regime. This result can be proved rigorously as follows.

By the classical result of detection in the Gaussian sequence model (cf.~e.g.~\cite[Sec.~3.3.6]{IS03}), without sparsity, the optimal $\lambda$ for \prettyref{eq:htsparsef} is $\Theta(p)$, achieved by the $\chi^2$-test, namely, thresholding on $\fnorm{X}$.
Next we show that this is optimal even when $k=1$.
To see this, consider the prior where $M$ is a random permutation matrix, which is $1$-sparse by definition and $\fnorm{M}=p$ with probability one. By \prettyref{lmm:chi2}, the $\chi^2$-divergence between the null and the alternative is
\begin{align}
\chi^2(P_{X|H_0}\,\|\,P_{X|H_1})+1
= & ~ \expect{\exp\pth{\Iprod{M}{\tilde{M}}}} = \expect{\exp(S_p)}
\end{align}
where $S_p$ is the number of fixed points of a uniform random permutation over $p$ elements. Furthermore, 
it is well-known that (cf.~\cite[Section IV.4]{feller1})
$S_p$ converges in distribution to Poisson$(1)$ as $p\diverge$ and, furthermore, 
	$\prob{S_n = \ell} = \frac{1}{\ell!}\sum_{m=0}^{n-\ell} \frac{(-1)^m}{m!} \leq \frac{2}{\ell!}$ for any $\ell \geq 0$,
which is faster than any exponential tail. Therefore,
	by \cite[Theorem 1]{mgf.conv}, the moment generating function of $S_n$ converges to that of Poisson$(1)$ pointwise. In particular, $\expect{\exp(S_p)} \to e^{e-1}$ as $p\to\infty$. Hence the probability of error for testing is non-vanishing in view of \prettyref{eq:chi2tv}.

	\subsection{Localization and denoising}
	\label{sec:discuss2}
	
A problem that is closely related to detecting the presence of a sparse matrix is localization. That is, the goal is to identify the support of the mean or covariance matrix with high probability. Under the row/column-wise sparsity assumption, if we measure the signal strength by the minimum non-zero entrywise magnitude, then it is easy to show that entrywise thresholding attains the minimax rate and there is no computational issue. In contrast, in the submatrix model, achieving the optimal rate for localization is again computationally difficult as shown in \cite{Cai2017Submatrix} and \cite{HajekWuXu14} in the context of Gaussian noise model and the community detection model, respectively.

Denoising high-dimensional matrices with submatrix sparsity was studied in \cite{MW13}, where the goal is to estimate the mean matrix $M$ based on the noisy observation in \prettyref{eq:GLM}. It turns out the computational difficulty of attaining the optimal rates crucially depends on the loss function \cite[Section 5.2]{MW13b}. 
For instance, for Frobenius norm loss entrywise thresholding is rate-optimal, while achieving the optimal rate for the spectral norm loss is no easier than planted clique whenever $k=p^\alpha$ for any fixed $0<\alpha<1$. 
In contrast, as mentioned earlier, for the sparsity model studied in this paper, entrywise thresholding achieves the minimax rate simultaneously for both the Frobenius norm and the spectral norm losses \cite{CZ12}.


%% file: Sparse-Matrix-Testing-122717.bbl
\begin{thebibliography}{10}

\bibitem{alon2007testing}
N.~Alon, A.~Andoni, T.~Kaufman, K.~Matulef, R.~Rubinfeld, and N.~Xie.
\newblock Testing $k$-wise and almost $k$-wise independence.
\newblock In {\em Proceedings of the thirty-ninth annual ACM symposium on
  Theory of computing}, pages 496--505. ACM, 2007.

\bibitem{Alon98}
N.~Alon, M.~Krivelevich, and B.~Sudakov.
\newblock Finding a large hidden clique in a random graph.
\newblock {\em Random Structures and Algorithms}, 13(3-4):457--466, 1998.

\bibitem{arias2012detection}
E.~Arias-Castro, S.~Bubeck, and G.~Lugosi.
\newblock Detection of correlations.
\newblock {\em The Annals of Statistics}, 40(1):412--435, 2012.

\bibitem{ABL12}
E.~Arias-Castro, S.~Bubeck, and G.~Lugosi.
\newblock Detecting positive correlations in a multivariate sample.
\newblock {\em Bernoulli}, 21(1):209--241, 2015.

\bibitem{CCHZ08}
E.~Arias-Castro, E.~J. Cand\'{e}s, H.~Helgason, and O.~Zeitouni.
\newblock Searching for a trail of evidence in a maze.
\newblock {\em The Annals of Statistics}, 36(4):1726--1757, 2008.

\bibitem{CCP11}
E.~Arias-Castro, E.~J. Cand{\`e}s, and Y.~Plan.
\newblock Global testing under sparse alternatives: {ANOVA}, multiple
  comparisons and the higher criticism.
\newblock {\em The Annals of Statistics}, 39(5):2533--2556, 2011.

\bibitem{CDH05}
E.~Arias-Castro, D.~L. Donoho, and X.~Huo.
\newblock Near-optimal detection of geometric objects by fast multiscale
  methods.
\newblock {\em {IEEE} Transactions on Information Theory}, 51(7):2402--2425,
  2005.

\bibitem{arias2013community}
E.~Arias-Castro and N.~Verzelen.
\newblock Community detection in dense random networks.
\newblock {\em Ann. Statist.}, 42(3):940--969, 2014.

\bibitem{balakrishnan2011tradeoff}
S.~Balakrishnan, M.~Kolar, A.~Rinaldo, A.~Singh, and L.~Wasserman.
\newblock Statistical and computational tradeoffs in biclustering.
\newblock In {\em NIPS 2011 Workshop on Computational Trade-offs in Statistical
  Learning}, 2011.

\bibitem{berthet2013lowerSparsePCA}
Q.~Berthet and P.~Rigollet.
\newblock Complexity theoretic lower bounds for sparse principal component
  detection.
\newblock {\em Journal of Machine Learning Research: Workshop and Conference
  Proceedings}, 30:1046--1066, 2013.

\bibitem{BR12}
Q.~Berthet and P.~Rigollet.
\newblock Optimal detection of sparse principal components in high dimension.
\newblock {\em The Annals of Statistics}, 41(4):1780--1815, 2013.

\bibitem{BJ08b}
P.~J. Bickel and E.~Levina.
\newblock Covariance regularization by thresholding.
\newblock {\em The Annals of Statistics}, 36(6):2577--2604, 2008.

\bibitem{BI12}
C.~Butucea and Y.~I. Ingster.
\newblock Detection of a sparse submatrix of a high-dimensional noisy matrix.
\newblock {\em Bernoulli}, 19(5B):2652--2688, 2013.

\bibitem{CJJ11}
T.~T. Cai, J.~X. Jeng, and J.~Jin.
\newblock Optimal detection of heterogeneous and heteroscedastic mixtures.
\newblock {\em Journal of the Royal Statistical Society. Series B
  (Methodological)}, 73(5):629 -- 662, 2011.

\bibitem{Cai2017Submatrix}
T.~T. Cai, T.~Liang, and A.~Rakhlin.
\newblock Computational and statistical boundaries for submatrix localization
  in a large noisy matrix.
\newblock {\em The Annals of Statistics}, 45:1403--1430, 2017.

\bibitem{CM13}
T.~T. Cai and Z.~Ma.
\newblock Optimal hypothesis testing for high dimensional covariance matrices.
\newblock {\em Bernoulli}, 19(5B):2359--2388, 2013.

\bibitem{CMW12}
T.~T. Cai, Z.~Ma, and Y.~Wu.
\newblock Sparse {PCA}: Optimal rates and adaptive estimation.
\newblock {\em The Annals of Statistics}, 41(6):3074 -- 3110, 2013.

\bibitem{CMW13}
T.~T. Cai, Z.~Ma, and Y.~Wu.
\newblock Optimal estimation and rank detection for sparse spiked covariance
  matrices.
\newblock {\em Probability Theory and Related Fields}, 161(3-4):781--815, 2015.

\bibitem{Cai2014ShortSeg}
T.~T. Cai and M.~Yuan.
\newblock Minimax rate optimal detection of very short signal segments.
\newblock {\em arXiv preprint arXiv:1407.2812}, 2014.

\bibitem{CZ12}
T.~T. Cai and H.~H. Zhou.
\newblock Optimal rates of convergence for sparse covariance matrix estimation.
\newblock {\em The Annals of Statistics}, 40:2389--2420, 2012.

\bibitem{candes2011robust}
E.~J. Cand{\`e}s, X.~Li, Y.~Ma, and J.~Wright.
\newblock Robust principal component analysis?
\newblock {\em Journal of the ACM (JACM)}, 58(3):11, 2011.

\bibitem{chen2010tests}
S.~X. Chen, L.-X. Zhang, and P.-S. Zhong.
\newblock Tests for high-dimensional covariance matrices.
\newblock {\em Journal of the American Statistical Association},
  105(490):810--819, 2010.

\bibitem{Csiszar63}
I.~Csisz{\'a}r.
\newblock Eine informationstheoretische ungleichung und ihre anwendung auf den
  beweis der ergodizitat von markoffschen ketten.
\newblock {\em Publ. Math. Inst. Hungar. Acad. Sci., Ser. A}, 8:85--108, 1963.

\bibitem{Davidson01}
K.~Davidson and S.~Szarek.
\newblock {\em Handbook on the Geometry of Banach Spaces}, volume~1, pages
  317--366.
\newblock Elsevier Science, 2001.

\bibitem{DJ.HC}
D.~L. Donoho and J.~Jin.
\newblock Higher criticism for detecting sparse heterogeneous mixtures.
\newblock {\em The Annals of Statistics}, 32(3):962--994, 2004.

\bibitem{fan2015estimation}
J.~Fan, P.~Rigollet, and W.~Wang.
\newblock Estimation of functionals of sparse covariance matrices.
\newblock {\em The Annals of statistics}, 43(6):2706, 2015.

\bibitem{FHT03}
A.~A. Fedotov, P.~Harremo{\"e}s, and F.~Tops{\o}e.
\newblock Refinements of {Pinsker's} inequality.
\newblock {\em Information Theory, IEEE Transactions on}, 49(6):1491--1498,
  2003.

\bibitem{Feldman2012statAlg}
V.~Feldman, E.~Grigorescu, L.~Reyzin, S.~Vempala, and Y.~Xiao.
\newblock Statistical algorithms and a lower bound for detecting planted
  cliques.
\newblock In {\em Proceedings of the forty-fifth annual ACM symposium on Theory
  of computing}, pages 655--664. ACM, 2013.

\bibitem{feller1}
W.~Feller.
\newblock {\em An Introduction to Probability Theory and Its Applications:
  Volume I}.
\newblock John Wiley \& Sons New York, 1968.

\bibitem{gao2017sparse}
C.~Gao, Z.~Ma, and H.~H. Zhou.
\newblock Sparse {CCA: A}daptive estimation and computational barriers.
\newblock {\em The Annals of Statistics}, 45(5):2074--2101, 2017.

\bibitem{matrix.computation}
G.~H. Golub and C.~F. Van~Loan.
\newblock {\em Matrix {C}omputations}.
\newblock Johns Hopkins University Press, Baltimore, MD, USA, 3rd edition,
  1996.

\bibitem{HWX14}
B.~Hajek, Y.~Wu, and J.~Xu.
\newblock Computational lower bounds for community detection on random graphs.
\newblock {\em Conference on Learning Theory (COLT)}, 2015.
\newblock arxiv:1406.6625.

\bibitem{HajekWuXu14}
B.~Hajek, Y.~Wu, and J.~Xu.
\newblock Computational lower bounds for community detection on random graphs.
\newblock In {\em Proceedings of COLT 2015}, pages 899--928, 2015.

\bibitem{HWX14b}
B.~Hajek, Y.~Wu, and J.~Xu.
\newblock Achieving exact cluster recovery threshold via semidefinite
  programming.
\newblock {\em IEEE Transactions on Information Theory}, 62(5):2788--2797,
  2016.

\bibitem{Hazan2011Nash}
E.~Hazan and R.~Krauthgamer.
\newblock How hard is it to approximate the best nash equilibrium?
\newblock {\em SIAM Journal on Computing}, 40(1):79--91, 2011.

\bibitem{Hoeffding63}
W.~Hoeffding.
\newblock Probability inequalities for sums of bounded random variables.
\newblock {\em Journal of the American Statistical Association},
  58(301):13--30, 1963.

\bibitem{Ingster97}
Y.~I. Ingster.
\newblock On some problems of hypothesis testing leading to infinitely
  divisible distributions.
\newblock {\em Mathematical Methods of Statistics}, 6(1):47 -- 69, 1997.

\bibitem{IS03}
Y.~I. Ingster and I.~A. Suslina.
\newblock {\em Nonparametric {G}oodness-of-fit {T}esting under {Gaussian}
  {M}odels}.
\newblock Springer, New York, NY, 2003.

\bibitem{Jer92}
M.~Jerrum.
\newblock Large cliques elude the {M}etropolis process.
\newblock {\em Random Structures \& Algorithms}, 3(4):347--359, 1992.

\bibitem{karoui2008operator}
N.~E. Karoui.
\newblock Operator norm consistent estimation of large-dimensional sparse
  covariance matrices.
\newblock {\em The Annals of Statistics}, pages 2717--2756, 2008.

\bibitem{klopp2015estimation}
O.~Klopp and A.~B. Tsybakov.
\newblock Estimation of matrices with row sparsity.
\newblock {\em Problems of Information Transmission}, 51(4):335--348, 2015.

\bibitem{mgf.conv}
W.~Kozakiewicz.
\newblock On the convergence of sequences of moment generating functions.
\newblock {\em Annals of Mathematical Statistics}, 18(1):61--69, 1947.

\bibitem{Lecam86}
L.~Le~Cam.
\newblock {\em Asymptotic Methods in Statistical Decision Theory}.
\newblock Springer-Verlag, New York, NY, 1986.

\bibitem{lounici2011oracle}
K.~Lounici, M.~Pontil, S.~Van De~Geer, and A.~B. Tsybakov.
\newblock Oracle inequalities and optimal inference under group sparsity.
\newblock {\em The Annals of Statistics}, 39(4):2164--2204, 2011.

\bibitem{MW13b}
Z.~Ma and Y.~Wu.
\newblock Computational barriers in minimax submatrix detection.
\newblock {\em The Annals of Statistics}, 43(3):1089--1116, 2015.

\bibitem{MW13}
Z.~Ma and Y.~Wu.
\newblock Volume ratio, sparsity, and minimaxity under unitarily invariant
  norms.
\newblock {\em IEEE Transactions on Information Theory}, 61(12):6939 -- 6956,
  Dec 2015.

\bibitem{onatski2013asymptotic}
A.~Onatski, M.~J. Moreira, and M.~Hallin.
\newblock Asymptotic power of sphericity tests for high-dimensional data.
\newblock {\em The Annals of Statistics}, 41(3):1204--1231, 2013.

\bibitem{onatski2014signal}
A.~Onatski, M.~J. Moreira, and M.~Hallin.
\newblock Signal detection in high dimension: The multispiked case.
\newblock {\em The Annals of Statistics}, 42(1):225--254, 2014.

\bibitem{RV07}
M.~Rudelson and R.~Vershynin.
\newblock Sampling from large matrices: {A}n approach through geometric
  functional analysis.
\newblock {\em Journal of the ACM (JACM)}, 54(4):21, 2007.

\bibitem{Tsybakov09}
A.~B. Tsybakov.
\newblock {\em Introduction to Nonparametric Estimation}.
\newblock Springer Verlag, New York, 2009.

\bibitem{vershynin2016high}
R.~Vershynin.
\newblock High-dimensional probability: An introduction with applications in
  data science.
\newblock Book draft, available at
  \url{https://www.math.uci.edu/~rvershyn/papers/HDP-book/HDP-book.pdf}.

\bibitem{wang2016statistical}
T.~Wang, Q.~Berthet, and R.~J. Samworth.
\newblock Statistical and computational trade-offs in estimation of sparse
  principal components.
\newblock {\em The Annals of Statistics}, 44(5):1896--1930, 2016.

\bibitem{yang2016rate}
D.~Yang, Z.~Ma, and A.~Buja.
\newblock Rate optimal denoising of simultaneously sparse and low rank
  matrices.
\newblock {\em The Journal of Machine Learning Research}, 17(1):3163--3189,
  2016.

\end{thebibliography}
